\documentclass[12pt,reqno]{amsart}

\usepackage{color}
\usepackage{amsmath, amsthm, amssymb}
\usepackage{amsfonts}
\usepackage[ansinew]{inputenc}
\usepackage[dvips]{epsfig}
\usepackage{graphicx}
\usepackage[english]{babel}
\usepackage{enumerate}
\usepackage{hyperref}
\theoremstyle{plain}
\newtheorem{thm}{Theorem}[section]
\newtheorem{cor}[thm]{Corollary}
\newtheorem{lem}[thm]{Lemma}
\newtheorem{prop}[thm]{Proposition}

\theoremstyle{definition}
\newtheorem{defi}[thm]{Definition}

\theoremstyle{remark}
\newtheorem{rem}[thm]{Remark}

\numberwithin{equation}{section}

\newcommand{\de}{\partial}

\newcommand{\fls}{(-\Delta)^s}
\newcommand{\R}{\mathbb{R}}

\newcommand{\N}{\mathbb{N}}

\newcommand{\average}{{\mathchoice {\kern1ex\vcenter{\hrule height.4pt
width 6pt depth0pt} \kern-9.7pt} {\kern1ex\vcenter{\hrule
height.4pt width 4.3pt depth0pt} \kern-7pt} {} {} }}

\def\R{\mathbb{R}}

\begin{document}

\title[Regularity for stable operators: parabolic equations]{Regularity theory for general stable operators: parabolic equations}

\author{Xavier Fernández-Real}

\address{ETH Z\"{u}rich, Department of Mathematics, Raemistrasse 101, 8092 Z\"{u}rich, Switzerland}

\email{xavierfe@math.ethz.ch}

\author{Xavier Ros-Oton}

\address{University of Texas at Austin, Department of Mathematics, 2515 Speedway, TX 78712 Austin, USA}

\email{ros.oton@math.utexas.edu}


\thanks{The first author is supported by ERC grant ``Regularity and Stability in Partial Differential Equations
(RSPDE)'' and a fellowship from ``Obra Social la Caixa''. The second author is supported by NSF grant DMS-1565186 and MINECO grant MTM2014-52402-C3-1-P (Spain).}

\begin{abstract}
We establish sharp interior and boundary regularity estimates for solutions to $\de_t u - L u = f(t, x)$ in $I\times \Omega$, with $I \subset \R$ and $\Omega \subset\R^n$. The operators $L$ we consider are infinitesimal generators of stable Lévy processes. These are linear nonlocal operators with kernels that may be very singular.

On the one hand, we establish interior estimates, obtaining that $u$ is $C^{2s+\alpha}$ in $x$ and $C^{1+\frac{\alpha}{2s}}$ in $t$, whenever $f$ is $C^{\alpha}$ in $x$ and $C^{\frac{\alpha}{2s}}$ in $t$. In the case $f\in L^\infty$, we prove that $u$ is $C^{2s-\epsilon}$ in $x$ and $C^{1-\epsilon}$ in $t$, for any $\epsilon > 0$.

On the other hand, we study the boundary regularity of solutions in $C^{1,1}$ domains. We prove that for solutions $u$ to the Dirichlet problem the quotient $u/d^s$ is Hölder continuous in space and time up to the boundary $\de\Omega$, where $d$ is the distance to $\de\Omega$. This is new even when $L$ is the fractional Laplacian.
\end{abstract}

\maketitle

\tableofcontents

\addtocontents{toc}{\protect\setcounter{tocdepth}{1}}  %això és perquè no surtin les subsections al tableofcontents

\section{Introduction and results}
The aim of this paper is to study the regularity of solutions to nonlocal parabolic equations
\begin{equation}
\label{eq.maineq}
\de_t u - L u = f(t, x),	
\end{equation}
where $L$ is a nonlocal operator of the form
\begin{equation}
\label{eq.L3}
L u(t, x) = \int_{\R^n} \bigl(u(t, x+y)+u(t, x-y)-2u(t, x)\bigr) \frac{a(y/|y|)}{|y|^{n+2s}}dy,
\end{equation}
with $s \in (0,1)$. Here, $a\in L^1(S^{n-1})$ is any nonnegative even function.

In fact, in order to allow the kernel of $L$ to be a singular measure we will be dealing with operators of the form
\begin{equation}
Lu(t, x) = \int_{S^{n-1}}\int_{-\infty}^\infty\bigr(u(t, x+\theta r)+ u(t,x-\theta r) - 2 u(t, x)\bigl) \frac{dr}{|r|^{1+2s}} d\mu(\theta),
\label{eq.L1}
\end{equation}
with the ellipticity conditions given by
\begin{equation}
0< \lambda \leq  \inf_{\nu \in S^{n-1}}\int_{S^{n-1}} |\nu \cdot \theta|^{2s}d\mu (\theta),~~~~~~~~~~\int_{S^{n-1}}d\mu\leq \Lambda < \infty
\label{eq.L2}
\end{equation}
for some constants $0<\lambda\leq \Lambda <\infty$. That is, we only require that the measure $\mu$, called the \emph{spectral measure}, is finite and cannot be supported in any proper subspace of $\R^n$. When $\mu$ is absolutely continuous, then $d\mu(\theta) = a(\theta)d\theta$ for $a\in L^1(S^{n-1})$, and we recover the expression \eqref{eq.L3}.

General operators of the form \eqref{eq.L1} arise as the infinitesimal generators of stable Lévy processes. These processes have been widely studied in both Probability and Analysis, and appear naturally in Mathematical Finance, Biology and Physics; see the introduction of \cite{RS} and references therein.

Important examples of stable operators to have in mind are the fractional Laplacian, $L = -\fls$,
\[
L v(x) = c_{n,s}\int_{\R^n} \bigl(v(x+y)+v(x-y)-2v(x)\bigr) \frac{dy}{|y|^{n+2s}},
\]
and the generator of $n$ independent $1$-dimensional symmetric stable Lévy processes,
\begin{equation}
\label{eq.deltas}
-L = (-\de_{x_1x_1})^s +\dots +(-\de_{x_nx_n})^s.
\end{equation}
In this case, the measure $\mu$ is a sum of $2n$ delta functions on the sphere. These two examples show the different degrees of regularity considered, from the very regular kernel in the fractional Laplacian ($\mu \equiv 1$), to the singular kernel given by the Dirac delta functions.

We will use parabolic Hölder seminorms. Given $\Omega\subset\R^n$, $I\subset\R$ and $\alpha,\beta\in(0,1)$, the parabolic seminorm $C^{\beta,\alpha}_{t,x}(I\times\Omega)$ is defined by
\begin{equation}
\label{eq.seminorm}
[u]_{{C^{\beta,\alpha}_{t,x}(I\times \Omega)}} := \sup_{\substack{t,t'\in I \\ x,x'\in\Omega}}\frac{| u(t,x)- u(t',x')|}{|t-t'|^\beta + |x-x'|^{\alpha}}.
\end{equation}
We will also denote
\[
[u]_{C^\alpha_x(I\times \Omega)} := \sup_{t'\in I} [u(t',\cdot)]_{C^\alpha(\Omega)},~~~~~~~~[u]_{C^\beta_t(I\times\Omega)} := \sup_{x'\in \Omega}[u(\cdot,x')]_{C^\beta(I)}.
\]

\subsection{Interior regularity}
We present here the main result regarding the interior regularity of solutions to nonlocal parabolic equations \eqref{eq.maineq}.

When the kernels in \eqref{eq.L3} are regular, interior regularity is fairly well understood; see for example \cite{JX, CD, CKS, CLK}. 
An important problem, however, is to understand what happens for singular kernels of the type \eqref{eq.L1}-\eqref{eq.L2}.

Important results in that direction have been recently obtained by Schwab-Silvestre \cite{SS} and by Kassmann-Schwab \cite{KS}. The results in \cite{SS,KS} allow kernels with no homogeneity and, more importantly, with $x$-dependence (with no regularity in $x$). For operators \eqref{eq.L1}, these results yield the H\"older continuity of solutions for wide classes of measures $\mu$. More precisely, the results in \cite{SS} yield the Hölder continuity of solutions to \eqref{eq.maineq} whenever the spectral measure $\mu$ is strictly positive on a set of positive measure; while the results of \cite{KS} do not assume the measure $\mu$ to be absolutely continuous, and apply also to the operator \eqref{eq.deltas}. Still, even in the case of translation invariant equations, the interior regularity for general stable operators was open.

In case of elliptic equations, this problem was recently solved in \cite{RS}, where the second author and Serra obtained sharp regularity estimates in H\"{o}lder spaces for \emph{all} translation invariant stable operators \eqref{eq.L1}-\eqref{eq.L2}. Here, we extend these estimates to the more general context of parabolic equations.

Our first main result is the following interior regularity estimate.
It essentially states that if $u_t-Lu=f\in C^{\frac{\alpha}{2s},\alpha}_{t,x}$ then $u$ is $C_t^{1+\frac{\alpha}{2s}}$ and $C^{\alpha+2s}_x$.
Notice that, even in the case $f=0$, the H\"older continuity of solutions is new.

\begin{thm} \label{thm.maina}
Let $s\in(0,1)$, and let $L$ be any operator of the form \eqref{eq.L1}-\eqref{eq.L2}. Let $u$ be any weak solution to
\begin{equation}
\label{eq.frac}
\de_t u-Lu = f \textrm{ in } (-1,0)\times B_1.
\end{equation}

Let $\alpha\in (0,1)$ be such that $\frac{\alpha}{2s}\in (0,1)$ and that $\alpha + 2s$ is not an integer. Let
\[
C_\alpha := \|u\|_{C^{\frac{\alpha}{2s},\alpha}_{t,x}((-1,0)\times \R^n)} + \|f\|_{C^{\frac{\alpha}{2s},\alpha}_{t,x}((-1,0)\times B_1)}.
\]
Then,
\begin{equation}
\label{eq.maina}
 \|u\|_{C^{1+\frac{\alpha}{2s}}_t \left(\left( -\frac{1}{2},0 \right)\times B_{1/2}\right)} +  \|u\|_{C^{2s+\alpha}_x\left(\left( -\frac{1}{2},0 \right)\times B_{1/2}\right)} \leq C C_\alpha,
\end{equation}
for some constant $C$ depending only on $n, s, \alpha$ and the ellipticity constants \eqref{eq.L2}.
\end{thm}

\begin{rem}
The previous expression \eqref{eq.maina} can equivalently be written as
\[
 \|u_t\|_{C^{\frac{\alpha}{2s}, \alpha}_{t,x} \left(\left( -\frac{1}{2},0 \right)\times B_{1/2}\right)} +  \|Lu\|_{C^{\frac{\alpha}{2s},\alpha}_{t,x}\left(\left( -\frac{1}{2},0 \right)\times B_{1/2}\right)} \leq C C_\alpha.
\]
This follows from \eqref{eq.maina} and the equation \eqref{eq.frac}.
\end{rem}

Notice that it is required that $u \in C^{\frac{\alpha}{2s}, \alpha}_{t,x}$ in all of $\R^n$ in order to have a $C^{1+\frac{\alpha}{2s},2s+\alpha}_{t,x}$ estimate in $B_{1/2}$. We show in Section~\ref{sec.7} that this is in fact necessary: we construct a solution $u$ to the homogeneous fractional heat equation, which satisfies $u\in C^{\frac{\alpha}{2s}-\epsilon, \alpha}_{t,x}((-1,0)\times\R^n)$ but $u\notin C_t^{1+\frac{\alpha}{2s}}\left(\left(-\frac{1}{2},0\right)\times B_{1/2}\right)$.

The spatial regularity requirements, $u\in C^\alpha_x$ in $(-1,0)\times\R^n$, can be relaxed if the kernel of the operator is regular; see Corollary \ref{cor.main3}.

When the right hand side in \eqref{eq.frac} is $f\in L^\infty$, we establish the following.

\begin{thm}
\label{thm.mainb}
Let $s\in(0,1)$, and let $L$ be any operator of the form \eqref{eq.L1}-\eqref{eq.L2}. Let $u$ be any weak solution to \eqref{eq.frac}. Let $\epsilon > 0$ and
\[
C_0 := \|u\|_{L^\infty((-1,0)\times \R^n)} + \|f\|_{L^\infty((-1,0)\times B_1)}.
\]
Then,
\[
 \|u\|_{C^{1-\frac{\epsilon}{2s}}_t \left(\left( -\frac{1}{2},0 \right)\times B_{1/2}\right)} +  \|u\|_{C^{2s-\epsilon}_x\left(\left( -\frac{1}{2},0 \right)\times B_{1/2}\right)} \leq C C_0,
\]
for some constant $C$ depending only on $n, s, \epsilon$ and the ellipticity constants \eqref{eq.L2}.
\end{thm}
See also Corollaries~\ref{cor.main1}, \ref{cor.main2}, \ref{cor.main4} for more consequences of Theorems~\ref{thm.maina} and \ref{thm.mainb}.

\subsection{Boundary regularity} We next present our boundary regularity results.

In the case of elliptic equations, the boundary regularity is quite well understood: see \cite{RS} for general stable operators in $C^{1,1}$ domains, and the results of Grubb \cite{GG14,GG15} for higher order estimates in case that $\Omega$ is $C^\infty$ and $a\in C^\infty(S^{n-1})$ in \eqref{eq.L3}.

Nonetheless, there are no similar boundary regularity results for nonlocal parabolic equations, not even when the operator $L$ is the fractional Laplacian.

Here, we extend the boundary regularity estimates of \cite{RS2,RS} to the context of parabolic equations. We state our results as local estimates for the following problem
\begin{equation}
\label{eq.main.fracdom}
  \left\{ \begin{array}{rcll}
  \de_t u - L u&=&f& \textrm{in }(-1,0)\times (\Omega\cap B_1)\\
  u&=&0& \textrm{in }(-1,0)\times B_1\setminus\Omega . \\
  \end{array}\right.
\end{equation}

First, we prove a $C^s_x$ regularity estimate up to the boundary. For the fractional Laplacian this could be deduced combining the heat kernel estimates from \cite{CKS} with known interior estimates. However, such precise heat kernel estimates are not known for more general stable operators.

\begin{prop}
\label{prop.mainboundary}
Let $s\in (0,1)$, let $\Omega$ be a $C^{1,1}$ domain, and let $L$ be an operator of the form \eqref{eq.L1}-\eqref{eq.L2}. Let $u$ be a weak solution to \eqref{eq.main.fracdom}. Then,
\begin{equation}
\|u\|_{C^{\frac{1}{2},s}_{t,x}\left(\left(-\frac{1}{2},0\right)\times B_{1/2}\right)}\leq C\left(\|f\|_{L^\infty((-1,0)\times (\Omega\cap B_1))} + \|u\|_{L^\infty((-1,0)\times\R^n)}\right),
\end{equation}
where $C$ depends only on $n, s, \Omega$ and the ellipticity constants \eqref{eq.L2}.
\end{prop}

In the next result, and throughout the rest of the paper, we denote
\[
d(x) :=\textrm{dist}(x, \R^n\setminus \Omega).
\]
Our second and main boundary regularity estimate is the following. This is new even when the operator $L$ is the fractional Laplacian.

\begin{thm}
\label{thm.mainboundary}
Let $s\in(0,1)$, let $\Omega$ be a $C^{1,1}$ domain and let $L$ be an operator of the form \eqref{eq.L1}-\eqref{eq.L2}. Let $u$ be a weak solution to \eqref{eq.main.fracdom} and
\[
C_0 = \|u\|_{L^\infty((-1,0)\times\R^n)}+\|f\|_{L^\infty((-1,0)\times (\Omega\cap B_1))}.
\]

Then, for any $\epsilon > 0$,
\begin{equation}
\|u\|_{C^{1-\epsilon}_t\left(\left(-\frac{1}{2},0\right)\times B_{1/2}\right)} + \left\| u/d^s \right\|_{C^{\frac{1}{2}-\frac{\epsilon}{2s},s-\epsilon}_x\left(\left(-\frac{1}{2},0\right)\times(\overline{\Omega}\cap B_{1/2})\right)} \leq C C_0.
\end{equation}
The constant $C$ depends only on $\epsilon$, $n$, $s$, $\Omega$  and the ellipticity constants \eqref{eq.L2}.
\end{thm}

\subsection{The Dirichlet problem}

Finally, we state the results from the previous subsection as a corollary regarding the Dirichlet problem. It is an immediate consequence of a combination of Theorem~\ref{thm.maina}, Proposition~\ref{prop.mainboundary}, and Theorem~\ref{thm.mainboundary}.

The Dirichlet problem for the nonlocal parabolic equations is
\begin{equation}\label{eq.fracheat_u0}
  \left\{ \begin{array}{rcll}
  \de_t u - L u&=&f& \textrm{in }\Omega,\ T > t > 0 \\
  u&=&0& \textrm{in } \R^n\setminus \Omega,\ T > t \geq 0, \\
  u(0,\cdot)&=&u_0& \textrm{in } \Omega,\ t = 0, \\
  \end{array}\right.
\end{equation}
where we consider again a domain $\Omega$, but now we also have to deal with an initial condition $u_0$, exterior conditions fixed in $\R^n\setminus\Omega$, and a time $T >0$.

The result reads as follows.

\begin{cor}
\label{cor.mainboundary}
Let $s\in(0,1)$, let $L$ be any operator of the form \eqref{eq.L1}-\eqref{eq.L2} and let $\Omega$ be a bounded $C^{1,1}$ domain. Suppose that $u$ is the weak solution to \eqref{eq.fracheat_u0}. Then,
\begin{equation}
\|u\|_{C^{1-\epsilon, s}_{t,x}\left(\left(t_0,T\right)\times \overline{\Omega}\right)} + \left\| u/d^s \right\|_{C^{\frac{1}{2}-\frac{\epsilon}{2s},s-\epsilon}_{t,x}\left(\left(t_0,T\right)\times\overline{\Omega}\right)} \leq C \left( \|u_0\|_{L^2(\Omega)} + \|f\|_{L^\infty((0,T)\times \Omega)} \right),
\end{equation}
for any $0<t_0< T$ and for all $\epsilon > 0$. The constant $C$ depends only on $\epsilon, n, s, \Omega, t_0, T$ and the ellipticity constants \eqref{eq.L2}.

Moreover, if $f\in C^{\frac{\alpha}{2s},\alpha}_{t,x}$, with $\alpha \in (0,s]$ such that $\alpha+2s$ is not an integer, then for any $K\Subset \Omega$ compact,
\begin{equation}
\label{eq.eqmainbd}
\|u\|_{C^{1+\frac{\alpha}{2s}}_{t}\left(\left(t_0,T\right)\times K\right)} + \|u\|_{C^{2s+\alpha}_{x}\left(\left(t_0,T\right)\times K\right)} \leq C \left( \|u_0\|_{L^2(\Omega)} + \|f\|_{C^{\frac{\alpha}{2s},\alpha}_{t,x}((0,T)\times \Omega)} \right).
\end{equation}
The constant $C$ depends only on  $\alpha, n, s, \Omega, K, t_0, T$ and the ellipticity constants \eqref{eq.L2}
\end{cor}

Notice that we require $\alpha \leq s$ in \eqref{eq.eqmainbd}. It turns out that, for general stable operators, solutions are not better than $C^{3s}_x$ inside $\Omega$; see \cite[Theorem 1.2]{RV} for a counterexample. Still, we prove here that this is not the case for time regularity, and show
\[
f\in C^\infty_t(\overline{\Omega})~\Rightarrow~u\in C^\infty_t(\overline{\Omega});
\]
see Corollary~\ref{cor.main2_3}.

\subsection{Ideas of the proofs}
To prove the interior and boundary regularity estimates we use blow-up arguments combined with Liouville-type theorems for parabolic nonlocal operators.

More precisely, in order to establish the interior regularity estimates we adapt the scaling method of Simon in \cite{S} to the context of nonlocal parabolic equations. We are then lead to a Liouville-type theorem in $(-\infty,0)\times\R^n$, which we prove by using the heat kernel of the operator, as in \cite{RS}.

On the other hand, to obtain the regularity up to the boundary for $u$ we adapt the methods of the second author and Serra in \cite{RS2} to parabolic equations. For this, we need to construct barriers with the appropriate behaviour near the boundary, which is done by combining the barriers of \cite{RS} with an eigenfunctions' decomposition of the solution to the parabolic Dirichlet problem in a bounded domain. 
Furthermore, to obtain the regularity up to the boundary for $u/d^s$ we first adapt the blow-up methods of \cite{RS} (based on the ideas in \cite{Ser}), and then combine them with the estimates for $u$ up to the boundary.

The paper is organised as follows. In Section \ref{sec.2} we prove Liouville-type theorem in the entire space, Theorem~\ref{thm.liouv}. In Section~\ref{sec.3} the interior regularity results are proved, Theorems~\ref{thm.maina} and \ref{thm.mainb}. In Section \ref{sec.4} we prove the $C^s_x$ regularity up to the boundary, Proposition~\ref{prop.mainboundary}, and deduce from it a Liouville-type theorem in the half space. Then, in Section \ref{sec.5} the main boundary regularity result is established, Theorem~\ref{thm.mainboundary}; and in Section \ref{sec.6} the Dirichlet problem is treated, thus proving Corollary~\ref{cor.mainboundary}. We end with some remarks on the sharpness of the estimates in Section~\ref{sec.7}.

\section{A Liouville-type theorem}
\label{sec.2}
In this section we prove the following result, a Liouville-type theorem for nonlocal parabolic equations.
\begin{thm}
\label{thm.liouv}
Let $s\in(0,1)$, and let $L$ be any operator of the form \eqref{eq.L1}-\eqref{eq.L2}. Let $u$ be any weak solution of
\[
\de_tu-Lu = 0~~~\textrm{in } (-\infty,0)\times\R^n
\]
such that
\[
\|u(t, \cdot)\|_{L^\infty(B_R)}\leq C\left(R^\gamma+1\right) \textrm{ for } R \geq |t|^{\frac{1}{2s}},
\]
for some $\gamma < 2s$. Then $u$ is a polynomial in the $x$ variable of degree at most $\lfloor\gamma\rfloor$.
\end{thm}

To prove the above theorem we follow the ideas of \cite{RS} for the elliptic problem.

We denote $p(t, x)$ the heat kernel associated to the operator $L$. Note that, by the scaling property of operators, we have
\[
p(t,x) = t^{-\frac{n}{2s}}p(1,xt^{-\frac{1}{2s}}).
\]

The following proposition is an immediate consequence of \cite[Proposition~2.2]{RS}.

\begin{prop}[\cite{RS}] Let $s\in(0,1)$ and let $L$ be any operator of the form \eqref{eq.L1}-\eqref{eq.L2}. Let $p(t, x)$ be the heat kernel associated to $L$. Then,

\begin{enumerate}[(a)]
\item For all $\delta \in (0,2s)$, $t > 0$,
\begin{equation}
\label{eq.liouv1}
\int_{\R^n}\left(1+|x|^{2s-\delta}\right) p(t, x)dx \leq C\left(1+t^{\frac{2s-\delta}{2s}}\right).
\end{equation}

\item Moreover, for $t > 0$,
\begin{equation}
\label{eq.liouv2}
[p(t, x)]_{C^{0,1}_x(\R^n)} \leq Ct^{-\frac{n+1}{2s}}
\end{equation}

\end{enumerate}
The constant $C$ depends only on $n$, $s$, $\delta$, and the ellipticity constants \eqref{eq.L2}.
\end{prop}

We can now prove the Liouville-type theorem.

\begin{proof}[Proof of Theorem \ref{thm.liouv}] The proof is parallel to the one done in \cite[Proposition 2.2]{RS} for the elliptic problem.

For every $\rho \geq 1$ define
\[
v(t, x) = \rho^{-\gamma} u(\rho^{2s} t, \rho x).
\]

It is easy to check $\de_t v -L v = 0$ in $(-\infty,0)\times\R^n$, and moreover, for $R \geq |t|^{\frac{1}{2s}}$,
\begin{equation}
\|v(t,\cdot)\|_{L^\infty(B_R)} = \rho^{-\gamma}\|u(\rho^{2s} t, \cdot)\|_{L^\infty(B_{\rho R})} \leq \rho^{-\gamma}C\left( 1+ (\rho R)^\gamma\right) \leq C\left( 1+ R^\gamma\right) \textrm{ for }\rho \geq 1.
\end{equation}

Now, we can use that, for $t \in (-1,0)$
\begin{equation}
\label{eq.liouv3}
v(t, x) \equiv v(-2,x)*p(2+t,x),
\end{equation}
where the convolution is only in the $x$ variable. Notice that when $t\in(-1,0)$, then $p(2+t,x)$ fulfils the same bounds as $p(1,x)$ in \cite{RS} with maybe different constants, thanks to \eqref{eq.liouv1}-\eqref{eq.liouv2}. Therefore, the rigorous proof of the equality \eqref{eq.liouv3} is the same as in \cite{RS}.

We now consider the function $v(t, x)$ for $t\in (-1,0)$. We know that $v(t, x) \leq C(|x|^\gamma+1)$, and we want to show
\[
[v]_{C^\xi_x((-1,0)\times B_1)} \leq C
\]
for some $\xi > 0$ and $C$ depending only on $n$, $\lambda$, $\Lambda$ and $\gamma$. To do so, let $x, x'\in B_1$, with $x\neq x'$, and let $t\in(-1,0)$. Then, using \eqref{eq.liouv3},
\begin{align*}
|v(t, x)-v(t, x')| & = \left| \int_{\R^n} (p(2+t, x-y)-p(2+t, x'-y))v(-2, y)dy \right|\\
& \leq \left|\int_{|y|\leq M} (p(2+t, x-y)-p(2+t, x'-y))v(-2, y)dy\right|+  \\
& ~~~+ 2\sup_{x\in B_1} \left|\int_{|y|\geq M} (p(2+t, x-y)v(-2, y)dy\right|.
\end{align*}

The first term in the sum can be bounded by
\[
\int_{|y|\leq M} (p(2+t, x-y)-p(2+t, x'-y))v(-2, y)dy \leq C M^{n+\gamma} |x-x'|,
\]
using \eqref{eq.liouv2} and the bound on $v(t, x)$. The second term is bounded again using the bound on $v(t, x)$ and \eqref{eq.liouv1} with $\delta = \frac{1}{2} (2s-\gamma)>0$,
\[
\left|\int_{|y|\geq M} p(2+t, x-y)v(-2, y)dy\right| \leq CM^{-\delta}.
\]
Thus, we have
\[
|v(t, x)-v(t, x')| \leq C M^{n+\gamma}|x-x'|+CM^{-\delta},\textrm{ for any } t\in (-1,0).
\]
Choosing
\[
M= |x-x'|^{-\xi/\delta},~~ ~\textrm{ with } 1-\frac{(n+\gamma)\xi}{\delta} = \xi ~ \Rightarrow~ \xi = \frac{2s-\gamma}{2s+2n+\gamma} > 0,
\]
we obtain
\[
[v]_{C^\xi_x((-1,0)\times B_1)} \leq C.
\]

Equivalently, for any $\rho \geq 1$,
\[
[u]_{C^\xi_x((-\rho^{2s},0)\times B_\rho)} \leq C\rho^{\gamma-\xi}.
\]

Let us now define the following incremental quotient function, for $h\in\R^n$ fixed,
\[
u_h^\xi (t, x) := \frac{u(t, x+h)-u(t, x)}{|h|^\xi},
\]
which, for any $t\in(-1,0)$, satisfies
\[
u_h^\xi (t, x) \leq C|x|^{\gamma-\xi}, \textrm{ for } |x| \geq 1.
\]
Now repeating the previous argument replacing $u$ by $u_h^\xi$, and $\gamma$ by $\gamma-\xi$, one gets $[u]_{C_x^{2\xi}((-R^{2s},0)\times B_R)}\leq CR^{\gamma-2\xi}$. We are using that after the previous step, the new $\xi' = \frac{2s-\gamma+\xi}{2s+2n+\gamma-\xi} > \xi$, so that we can take $\xi$ instead of $\xi'$. Iterating this procedure, after $N$ steps,
\[
[u]_{C^{N\xi}_x((-R^{2s},0)\times B_R)} \leq CR^{\gamma-N\xi}.
\]

Taking $N$ as the least integer such that $N\xi > \gamma$ and letting $R\to \infty$, we finally obtain
\[
[u]_{C^{N\xi}_x((-\infty,0)\times\R^n)}= 0,
\]
which implies that for each $t\in (-\infty,0)$, $u(t, x)$ is a polynomial on $x$ of degree at most $\lfloor \gamma \rfloor \leq \lfloor 2s\rfloor$.

Finally, $Lu = 0$ for all $(t, x) \in (-\infty,0)\times\R^n$, and thus $\de_t u = 0$, so that $u$ is constant with respect to $t$.
\end{proof}

\section{Interior regularity}
\label{sec.3}
In this section we prove the main results regarding the interior regularity of the solutions, Theorems \ref{thm.maina} and \ref{thm.mainb}. We first present a short subsection introducing the seminorms that we are going to use in the proofs.

\subsection{Parabolic Hölder seminorms} Many times we will implicitly use that
\begin{equation}
\label{eq.equivseminorm}
[u]_{{C^{\beta,\alpha}_{t,x}(I\times \Omega)}} \sim [u]_{C^\beta_t(I\times\Omega)}+[u]_{C^\alpha_x(I\times \Omega)},
\end{equation}
in the sense that these two seminorms are equivalent.

The definition in \eqref{eq.seminorm} is not enough for the cases considered in this paper; we need to introduce higher order parabolic Hölder seminorms. For $s\in(0,1)$, and for given $\alpha \in (0,1)$, we define
\[
\beta = \frac{\alpha}{2s}.
\]

Moreover, throughout the section we assume that $\alpha$ is such that $\beta\in (0,1)$ ($\alpha < 2s$), and that $\alpha +2s$ not an integer. Let $\nu = \lfloor 2s+\alpha \rfloor$ and $I\times\Omega$ be a bounded domain with $I\subset(-\infty,0],~\Omega \subset\R^n$. We define the following parabolic Hölder seminorm
\begin{equation}
\label{eq.notation1}
 [u]^{(1+\beta,2s+\alpha)}_{I\times \Omega} =
  \begin{cases}
      [\de_t u]_{C_{t,x}^{\beta,\alpha}(I\times \Omega)} +  [u]_{C^{2s+\alpha}_x(I\times \Omega)}   \hfill & \text{ if } \nu = 0, \\[6mm]
       [\de_t u]_{C_{t,x}^{\beta,\alpha}(I\times \Omega)} +  [\nabla_x u]_{C^{\frac{2s+\alpha-1}{2s},2s+\alpha-1}_{t,x}(I\times \Omega)}   \hfill & \text{ if } \nu = 1, \\[6mm]
       [\de_t u]_{C_{t,x}^{\beta,\alpha}(I\times \Omega)} +[\nabla_x u]_{C_t^{\frac{2s+\alpha-1}{2s}}(I\times \Omega)} +  [D^2_x u]_{C^{\frac{2s+\alpha-2}{2s},2s+\alpha-2}_{t,x}(I\times \Omega)}   \hfill & \text{ if } \nu = 2.
  \end{cases}
\end{equation}

Notice that with this choice of norms we always have a good rescaling. That is, if $u_\rho (t, x) = u(\rho^{2s} t, \rho x)$ then
\begin{equation}
\label{eq.rescnorm}
[u_\rho]_{I\times\Omega}^{(1+\beta,2s+\alpha)} = \rho^{2s+\alpha}[u]_{(\rho^{-2s}I)\times (\rho^{-1}\Omega)}^{(1+\beta,2s+\alpha)}.
\end{equation}

The previous definition will be useful to prove Theorem \ref{thm.maina}, but for Theorem \ref{thm.mainb} we need a definition for different indices. Namely, we will denote $\nu := \lceil 2s \rceil - 1$, $\epsilon > 0$ such that $2s-\epsilon > \nu$, and

\begin{equation}
\label{eq.notation2}
 [u]^{\left(1-\frac{\epsilon}{2s},2s-\epsilon\right)}_{I\times \Omega} =
  \begin{cases}
      [u]_{C_{t,x}^{1-\frac{\epsilon}{2s},2s-\epsilon}(I\times \Omega)} \hfill & \text{ if } \nu = 0, \\[6mm]
        [u]_{C_{t}^{1-\frac{\epsilon}{2s}}(I\times \Omega)} +  [\nabla_x u]_{C^{\frac{2s-\epsilon-1}{2s},2s-\epsilon-1}_{t,x}(I\times \Omega)}   \hfill & \text{ if } \nu = 1.
  \end{cases}
\end{equation}
Note that we still have a good rescaling, i.e., for $u_\rho (t, x) = u(\rho^{2s} t, \rho x)$ then
\begin{equation}
\label{eq.rescnorm_2}
[u_\rho]^{\left(1-\frac{\epsilon}{2s},2s-\epsilon\right)}_{I\times \Omega}= \rho^{2s-\epsilon}[u]^{\left(1-\frac{\epsilon}{2s},2s-\epsilon\right)}_{(\rho^{-2s}I)\times (\rho^{-1}\Omega)}.
\end{equation}

The full norm is defined by
\begin{align*}
\|u\|^{1+\beta,2s+\alpha}_{(I\times \Omega)} & =  \|u\|_{C^{1,\nu}_{t,x}(I\times \Omega)}+ [u]^{(1+\beta,2s+\alpha)}_{I\times \Omega}\\
& :=  \|\de_t u \|_{L^\infty(I\times \Omega)} + \sum_{|\psi| \leq \nu} \|D^\psi_x u \|_{L^\infty(I\times \Omega)} + [u]^{(1+\beta,2s+\alpha)}_{I\times \Omega},
\end{align*}
where we have also defined the $\|u\|_{C^{1,\nu}_{t,x}(I\times \Omega)}$ norm. The definition of $\|u\|^{1-\frac{\epsilon}{2s},2s-\epsilon}_{(I\times \Omega)}$ is analogous.

An interpolation inequality can be proved for these norms: for any $\kappa> 0$ we have
\begin{equation}
\label{eq.interp}
\|u\|_{C^{1,\nu}_{t,x}(I\times \Omega)} \leq \kappa [u]^{(1+\beta,2s+\alpha)}_{I\times \Omega} + C\|u\|_{L^\infty(I\times \Omega)},
\end{equation}
for some constant $C$ depending only on $\kappa$, $n$, $s$, $\alpha$ and $\beta$. Analogously,
\begin{equation}
\label{eq.interp2}
\|u\|_{C^{0,\nu}_{t,x}((-1,0)\times B_1)}\leq \kappa [u]^{(1-\frac{\epsilon}{2s},2s-\epsilon)}_{((-1,0)\times B_1)} + C\|u\|_{L^\infty((-1,0)\times B_1)},
\end{equation}
for $C$ constant now depending only on $\kappa$, $n$, $s$ and $\epsilon$.

To show \eqref{eq.interp}-\eqref{eq.interp2}, it is enough to use the equivalence of seminorms \eqref{eq.equivseminorm} and classical interpolation inequalities in Hölder spaces.

Finally, another result we will need is the following inequality in bounded domains: for any given $\eta \in C^\infty(\R^n)$, and $u = u(t, x)$, then
\begin{equation}
\label{eq.1}
[\eta u]_{I\times \Omega}^{(1+\beta,2s+\alpha)} \leq C\left(\|\eta\|_{C^{2s+\alpha}_x( \Omega)} \|u\|_{C^{1,\nu}_{t,x}(I\times \Omega)} + [u]^{(1+\beta,2s+\alpha)}_{I\times \Omega} \|\eta\|_{C^\nu_x( \Omega)}\right),
\end{equation}
for some constant $C$ depending only on $n$, $s$, $\alpha$, $\beta$, $I$ and $\Omega$. To see it, use the analogous inequality for Hölder spaces and the equivalence of seminorms \eqref{eq.equivseminorm}. Similarly, one finds
\begin{equation}
\label{eq.1.1}
[\eta u]_{I\times \Omega}^{(\beta,\alpha)} \leq C\left(\|\eta\|_{C^{\alpha}_x( \Omega)} \|u\|_{L^\infty (I\times \Omega)} + [u]^{(\beta,\alpha)}_{I\times \Omega} \|\eta\|_{L^\infty( \Omega)}\right),
\end{equation}
for some constant $C$ depending only on $n$, $s$, $\alpha$, $\beta$, $I$ and $\Omega$.

\subsection{Proof of Theorem \ref{thm.maina}}

Now that we have introduced the notation, let us proceed to prove the results regarding the interior regularity of solutions to nonlocal parabolic equations.

To begin with, the following lemma will give us a tool to study the convergence of functions in the proofs of Proposition \ref{prop.main} and Proposition \ref{prop.main2} below.

\begin{lem}
\label{lem.1}
Let $s\in(0,1)$, $\lambda, \Lambda > 0$ fixed constants, and let $(L_k)_{k\in\N}$ be a sequence of operators of the form \eqref{eq.L1}-\eqref{eq.L2}. Let $(u_k)_{k\in\N}$ and $(f_k)_{k\in\N}$ be sequences of functions satisfying in the weak sense
\[
\de_t u_k-L_k u_k = f_k \textrm{ in } I\times K
\]
for a given bounded interval $I\subset (-\infty,0]$ and a bounded domain $K \subset\R^n$.

Assume that $L_k$ have spectral measures $\mu_k$ converging to a spectral measure $\mu$. Let $L$ be the operator associated to $\mu$ (weak limit of $L_k$), and suppose that, for some functions $u$ and $f$ the following hypotheses hold:
\begin{enumerate}
\item $u_k\to u$ uniformly in compact sets of $(-\infty,0]\times\R^n$,
\item $f_k \to f$ uniformly in $I\times K$,
\item $\sup_{t\in I}|u_k(t, x)| \leq C\left(1+|x|^{2s-\epsilon}\right)$ for some $\epsilon > 0$, and for all $x\in \R^n$.
\end{enumerate}

Then, $u$ satisfies
\[
\de_tu-Lu = f \textrm{ in } I\times K
\]
in the weak sense.
\end{lem}
\begin{proof}
We have that
\[
\int_{I\times\R^{n}} u_k (-\de_t \eta -L_k \eta) = \int_{I\times K} f_k\eta,~~~\textrm{ for all } \eta \in C_c^\infty(I\times K).
\]

On the other hand, since $|\eta(x+y)+\eta(x-y)-2\eta(x)|\leq C\min\{1,|y|^2\}$, by the dominated convergence theorem we obtain that $L_k \eta \to L\eta$ uniformly over compact subsets of $I\times \R^n$.

Moreover, $\eta$ has support in $K$, which yields $|L_k\eta(x)|\leq C (1+|x|^{n+2s})^{-1}$. Combining this with the growth of $u_k$ (see hypothesis (3)) we get that $|u_k(-\de_t\eta-L_k\eta)|\leq C(1+|x|^{n+\epsilon})^{-1}$, and therefore, by the dominated convergence theorem
\[
\int_{I\times\R^{n}}u_k(-\de_t\eta-L_k\eta)\to \int_{I\times\R^{n}} u (-\de_t\eta-L\eta),~~~\textrm{ for all } \eta \in C_c^\infty(I\times K).
\]

Since it is clear that
\[
\int_{I\times K} f_k\eta \to \int_{I\times K} f \eta,
\]
then we have that the limit $u$ is a weak solution to the equation
\[
\de_t u-Lu = f \textrm{ in } I\times K,
\]
as desired.
\end{proof}

Before proceeding to prove Theorem~\ref{thm.maina}, let us first show the following.

\begin{prop}
\label{prop.main}
Let $s\in(0,1)$, and let $L$ be an operator of the form \eqref{eq.L1}-\eqref{eq.L2}. Let $\alpha \in (0,1)$, such that $\beta = \frac{\alpha}{2s}\in(0,1)$ and $\alpha+2s$ is not an integer, and let $\nu = \lfloor 2s+\alpha\rfloor$. Assume that $u\in C^\infty_c((-\infty,0]\times\R^n)$ satisfies
\[
\de_t u-Lu = f \textrm{ in } (-1,0)\times B_1
\]
with $f\in C_{t,x}^{\beta,\alpha}((-1,0)\times B_1)$. Then, for any $\delta > 0$ we have
\begin{equation}
\label{eq.ineq1}
[u]^{(1+\beta,2s+\alpha)}_{((-2^{-2s},0)\times B_{1/2})} \leq \delta [u]^{(1+\beta,2s+\alpha)}_{((-1,0)\times \R^n)} + C \left(\|u\|_{C^{1,\nu}_{t,x} ((-1,0)\times B_1)} + [f]_{C_{t,x}^{\beta,\alpha}((-1,0)\times B_1)} \right),
\end{equation}
where the constant $C$ depends only on $\delta$, $n$, $s$, $\alpha$ and the ellipticity constants \eqref{eq.L2}.
\end{prop}

\begin{proof}

Let us argue by contradiction. Suppose that for a given $\delta > 0$ the estimate does not hold for any constant $C$, so that for each $k \in \N$ we have that there exist functions $w_k\in C^\infty_c((-\infty,0]\times\R^n), f_k\in C_{t,x}^{\beta,\alpha}((-1,0)\times B_1)$, and operators $L_k$ of the form \eqref{eq.L1}-\eqref{eq.L2} such that $\de_t w_k-L_k w_k = f_k$ in $(-1,0)\times B_1$ and
\begin{equation}
\label{eq.ineqk}
[w_k]^{(1+\beta,2s+\alpha)}_{((-2^{-2s},0)\times B_{1/2})} > \delta [w_k]^{(1+\beta,2s+\alpha)}_{((-1,0)\times \R^n)} + k \left(\|w_k\|_{C^{1,\nu}_{t,x} ((-1,0)\times B_1)} + [f_k]_{C_{t,x}^{\beta,\alpha}((-1,0)\times B_1)} \right).
\end{equation}

In order to find the contradiction we will follow four steps.
\\[0.5cm]
%************STEP 1************************
{\bf Step 1: The blow-up parameter, $\rho_k$.} We will need to separate three different cases, according to the value of $\nu$.

$\bullet$ {\it Case $\nu = 0$.} The seminorm in this case is
\[
 [w_k]^{(1+\beta,2s+\alpha)}_{(-2^{-2s},0)\times B_{1/2}} = [\de_t w_k]_{C_{t,x}^{\beta,\alpha}((-2^{-2s},0)\times B_{1/2})} +  [w_k]_{C^{2s+\alpha}_x((-2^{-2s},0)\times B_{1/2})}  ,
\]
and by definition, we can choose $x_k,y_k\in B_{1/2},~t_k,s_k\in (-2^{-2s},0)$ such that
\begin{equation}
\label{eq.ineqk2.0}
\frac{1}{4}[w_k]^{(1+\beta,2s+\alpha)}_{((-2^{-2s},0)\times B_{1/2})} < \frac{|\de_t w_k(t_k,x_k)-\de_t w_k(s_k,y_k)|}{|s_k-t_k|^\beta+|x_k-y_k|^\alpha}+\frac{| w_k(t_k,x_k) -  w_k(t_k,y_k)|}{|x_k-y_k|^{2s+\alpha}}.
\end{equation}

Define,
\[
\rho_k := |t_k-s_k|^{\frac{1}{2s}}+|x_k-y_k|,
\]
and now we claim that, up to possibly a new choice of $s_k$, we have
\begin{equation}
\label{eq.ineqk2.0.1}
\chi(n)[w_k]^{(1+\beta,2s+\alpha)}_{((-2^{-2s},0)\times B_{1/2})} < \frac{|\de_t w_k(t_k,x_k)-\de_t w_k(s_k,y_k)|}{\rho_k^\alpha}+\frac{| w_k(t_k,x_k) -  w_k(t_k,y_k)|}{\rho_k^{2s+\alpha}},
\end{equation}
for some small constant $\chi > 0$ depending only on $n$. Indeed, if in the equation \eqref{eq.ineqk2.0} the first term of the sum is greater than the second one, we are done. Otherwise, we can fix $s_k=t_k$ and obtain the desired result.

Therefore, we have
\begin{align*}
\chi(n)[w_k &]^{(1+\beta,2s+\alpha)}_{((-2^{-2s},0)\times B_{1/2})} <\\
<& \frac{2\|\de_t w_k\|_{L^\infty((-2^{-2s},0)\times B_{1/2})}}{\rho_k^{\alpha}}+\frac{2\|w_k\|_{L^\infty((-2^{-2s},0)\times B_{1/2})}}{\rho_k^{2s+\alpha}}\\
\leq & \left(\frac{1}{\rho_k^\alpha} + \frac{1}{\rho_k^{2s+\alpha}}\right) \frac{[w_k]^{(1+\beta,2s+\alpha)}_{((-2^{-2s},0)\times B_{1/2})}}{k} ,
\end{align*}
where in the last inequality we are using \eqref{eq.ineqk}. We finally obtain that $\rho_k\to 0$ as $k\to \infty$, since
\[
\chi(n)  k\leq \frac{1}{\rho_k^\alpha} + \frac{1}{\rho_k^{2s+\alpha}}.
\]

$\bullet$ {\it Case $\nu = 1$.} Proceed as before, by choosing $x_k,y_k\in B_{1/2},~t_k,s_k\in (-2^{-2s},0)$ such that
\begin{equation}
\label{eq.ineqk2}
\frac{1}{2}[w_k]^{(1+\beta,2s+\alpha)}_{((-2^{-2s},0)\times B_{1/2})} < \frac{|\de_t w_k(t_k,x_k)-\de_t w_k(s_k,y_k)|}{|s_k-t_k|^\beta+|x_k-y_k|^\alpha}+\frac{|\nabla_x w_k(t_k,x_k) - \nabla_x w_k(s_k,y_k)|}{|t_k-s_k|^{\frac{2s+\alpha-1}{2s}}+|x_k-y_k|^{2s+\alpha-1}}.
\end{equation}

Define $\rho_k$ as before. Now, it immediately follows that
\begin{equation}
\label{eq.ineqk2.0.2}
\chi(n)[w_k]^{(1+\beta,2s+\alpha)}_{((-2^{-2s},0)\times B_{1/2})} < \frac{|\de_t w_k(t_k,x_k)-\de_t w_k(s_k,y_k)|}{\rho_k^\alpha}+\frac{| \nabla_x w_k(t_k,x_k) -  \nabla_x w_k(s_k,y_k)|}{\rho_k^{2s+\alpha-1}},
\end{equation}
where we keep the same constant $\chi(n)> 0$ as in the previous case, by making it smaller if necessary. The same reasoning as before yields $\rho_k \to 0$ as $k\to \infty$.

$\bullet$ {\it Case $\nu = 2$}. A similar reasoning as in the case $\nu = 0$ yields that, for some constant $\chi(n)> 0$, there are $x_k,y_k\in B_{1/2},~t_k,s_k\in (-2^{-2s},0)$ and $\rho_k$ defined as before such that
\begin{align}
\label{eq.ineqk2.0.3}
\chi(n)[w_k &]^{(1+\beta,2s+\alpha)}_{((-2^{-2s},0)\times B_{1/2})} <
 \frac{|\de_t w_k(t_k,x_k)-\de_t w_k(s_k,y_k)|}{\rho_k^\alpha}\\&+\frac{| D^2_x w_k(t_k,x_k) -  D^2_x w_k(s_k,y_k)|}{\rho_k^{2s+\alpha-2}}\nonumber + \frac{| \nabla_x w_k(t_k,x_k) -  \nabla_x w_k(s_k,x_k)|}{\rho_k^{2s+\alpha-1}}.
\end{align}
Therefore, by the previous argument, $\rho_k \to 0$ as $k \to \infty$.
\\[0.5cm]
%****************STEP 2*****************
{\bf Step 2: The blow-up sequence.} We will proceed with a blow-up method. We begin by defining the following functions, where we will assume that $t_k \geq s_k$ (otherwise, we can swap them),
\[
v_k(t,x):= \frac{w_k(t_k+\rho_k^{2s} t,x_k+\rho_k x) - p_k(t, x)}{\rho_k^{2s+\alpha} [w_k]^{(1+\beta,2s+\alpha)}_{((-1,0)\times \R^n)}}.
\]
Here, $p_k(t, x)$ is a polynomial in $t$ of degree at most $1$ plus a polynomial in $x$ of degree at most $\nu$, such that
\begin{equation}
\label{eq.vkzero}
v_k(0,0) = \dots = D^\nu_x v_k(0,0) = \de_tv_k(0,0)= 0.
\end{equation}

First, notice that this function has a bounded $(1+\beta, 2s+\alpha)$-seminorm,
\begin{equation}
\label{eq.bdsemi}
[v_k]^{(1+\beta,2s+\alpha)}_{\left(\left(-\frac{1}{2}\rho_k^{-2s},0\right]\times \R^n \right)} \leq 1.
\end{equation}
Indeed, this follows simply by considering the scaling of the seminorm (see \eqref{eq.rescnorm}), and noticing that the seminorm of the polynomial $p_k$ is zero.

Secondly, we have uniform convergence towards $0$ of the following quantity for fixed $\tau \in (-1,0)$ and $h\in B_1$,
\begin{equation}
\label{eq.bdk}
|(\de_t-L_k)(v_k(t+\tau,x+h)-v_k(t, x))| \leq \frac{C(n)}{k} \to 0
\end{equation}
uniformly in $\left(-\frac{1}{2}\rho_k^{-2s},0\right)\times B_{\left(\frac{1}{2}\rho_k^{-1} - h\right)}$. Indeed,
\begin{align*}
 |(&\de_t-  L_k)(v_k(t+\tau,x+h)-v_k(t, x))| = \\
& = \frac{|\rho_k^{2s}\tau|^\beta+|\rho_k h|^\alpha}{\rho_k^{\alpha} [w_k]^{(1+\beta,2s+\alpha)}_{((-1,0)\times \R^n)}}\cdot \frac{|f_k(t_k+\rho_k^{2s} (t+\tau),x_k+\rho_k (x+h)) - f_k(t_k+\rho_k^{2s} t,x_k+\rho_k x)|}{|\rho_k^{2s}\tau|^\beta+|\rho_k h|^\alpha} \\
&\leq C(n)\frac{[f_k]_{C_{t,x}^{\beta,\alpha}((-1,0)\times B_1)}}{[w_k]^{(1+\beta,2s+\alpha)}_{((-1,0)\times \R^n)}}\leq \frac{C(n)}{k} \to 0,
\end{align*}
where in the last inequality we used \eqref{eq.ineqk}.

We now define the following points in the set $[-1,0]\times\overline{B_1}$,
\[
\xi_k = \left( \frac{s_k-t_k}{\rho_k^{2s}},\frac{y_k-x_k}{\rho_k} \right),~~~~
\xi_k^{(1)} = \left( \frac{s_k-t_k}{\rho_k^{2s}},0 \right),~~~~
\xi_k^{(2)} = \left( 0,\frac{y_k-x_k}{\rho_k} \right),
\]
and notice that
\begin{alignat*}{2}
\de_t v_k(\xi_k)= &\frac{\de_t w_k(s_k,y_k) - \de_t w_k(t_k,x_k)}{\rho_k^{\alpha} [w_k]^{(1+\beta,2s+\alpha)}_{((-1,0)\times \R^n)}}&&~~~\textrm{ for }\nu = 0,1,2,\\[4mm]
 v_k(\xi_k^{(2)})= &\frac{ w_k(t_k,y_k) -  w_k(t_k,x_k)}{\rho_k^{2s+\alpha} [w_k]^{(1+\beta,2s+\alpha)}_{((-1,0)\times \R^n)}}&& ~~~\textrm{ for } \nu = 0,
\\[4mm]
D^\nu_x v_k(\xi_k)= &\frac{D^\nu_x w_k(s_k,y_k) - D^\nu_x w_k(t_k,x_k)}{\rho_k^{2s+\alpha-\nu} [w_k]^{(1+\beta,2s+\alpha)}_{((-1,0)\times \R^n)}}&& ~~~\textrm{ for } \nu = 1,2,\\[4mm]
\nabla_x v_k(\xi_k^{(1)})= &\frac{\nabla_x w_k(s_k,x_k) - \nabla_x w_k(t_k,x_k)}{\rho_k^{2s+\alpha-1} [w_k]^{(1+\beta,2s+\alpha)}_{((-1,0)\times \R^n)}}&& ~~~\textrm{ for } \nu = 2,
\end{alignat*}
Thus, combining \eqref{eq.ineqk2.0.1}-\eqref{eq.ineqk2.0.2}-\eqref{eq.ineqk2.0.3} with \eqref{eq.ineqk} we obtain
\begin{alignat}{2}
\label{eq.nondeg}
|\de_t v_k(\xi_k)|+|v_k(\xi_k^{(2)})| >& \chi(n) \delta && ~~~~\textrm{ if }\nu = 0,\nonumber \\[4mm]
 |\de_t v_k(\xi_k)|+|\nabla_x v_k(\xi_k)| >& \chi(n) \delta && ~~~~ \textrm{ if } \nu = 1,
\\[4mm]
|\de_t v_k(\xi_k)|+|D^2_x v_k(\xi_k)| + |\nabla_x v_k(\xi_k^{(1)})| >& \chi(n) \delta && ~~~~\textrm{ if }\nu = 2.\nonumber
\end{alignat}

Notice that, up to a subsequence, $\xi_k$ converge to some $\xi\in [0,1]\times\overline{B_1}$ (and so do $\xi_k^{(1)}$ and $\xi_k^{(2)}$). From now on we restrict ourselves to this subsequence.
\\[0.5cm]
%**************************STEP 3**************************
{\bf Step 3. Convergence properties of the blow-up sequence.} Recall that we have a uniform bound on the seminorms of $v_k$, \eqref{eq.bdsemi}. Thus, we deduce that, up to subsequences, $v_k$ converges in $C^1_t$ and in $C^\nu_x$ to some function $v$ over compact subsets of $(-\infty,0]\times\R^n$. Indeed, this follows since the Hölder seminorms $[v_k]_{C^{1+\beta}_t\left(\left(-\frac{1}{2}\rho_k^{-2s},0\right]\times\R^n\right)}$, $[v_k]_{C^{2s+\alpha}_x\left(\left(-\frac{1}{2}\rho_k^{-2s},0\right]\times\R^n\right)}$ are uniformly bounded with respect to $k\in\N$, and the domains are expanding to $(-\infty,0]\times\R^n$.

We restrict ourselves to this subsequence, and obtain a limit function $v$ defined in $(-\infty,0]\times\R^n$ such that
\begin{equation}
\label{eq.viszero}
v(0,0) = \dots = D^\nu_x v(0,0) = \de_t v(0,0)= 0, ~~\textrm{and}~~~
[v]^{(1+\beta,2s+\alpha)}_{(-\infty,0)\times \R^n)} \leq 1.
\end{equation}
By \eqref{eq.nondeg} and the nice convergence in $C^{1,\nu}_{t,x}$, we get that $v$ cannot be constant.

From now on we want to consider the functions $v_k(t+\tau, x+h)-v_k(t, x)$ for fixed $\tau \in (-1,0), h\in B_1$. We want to compute an upper bound for $|v_k(t+\tau, x+h)-v_k(t, x)|$, for $t\in\left(-\frac{1}{2}\rho_k^{-2s}-\tau,0\right]$ and $x\in \R^n$. To do so we separate three cases again:

$\bullet$ {\it Case $\nu = 0$},
\begin{align*}
|v_k(t+\tau, x+h)-v_k(t, x)|& \leq \tau \sup_{\substack{t'\in(t+\tau, 0)\\x'\in B_{|x|+|h|}}} |\de_t v_k(t',x')| \\
&~~~~~~~~~~~~~~+ h^{2s+\alpha} \sup_{t'\in (t+\tau,0)}[v_k(t',\cdot)]_{C^{2s+\alpha}(B_{|x|+|h|})} \\
& \leq C (|x|^\alpha + |t|^\beta + 1)
\end{align*}
using the bounds on the seminorms of $v_k$ and \eqref{eq.vkzero}. The constant $C$ can depend on $\tau$ and $h$.

$\bullet$ {\it Case $\nu = 1$},
\begin{align*}
|v_k(t+\tau, x+h)-v_k(t, x)|\leq & \tau \sup_{\substack{t'\in(t+\tau, 0)\\x'\in B_{|x|+|h|}}} |\de_t v_k(t',x')| + h \sup_{\substack{t'\in(t+\tau, 0)\\x'\in B_{|x|+|h|}}}|\nabla_x v_k(t',x')|\\
 \leq & C(|x|^\alpha+|x|^{2s+\alpha-1} + |t|^\beta + |t|^{\frac{2s+\alpha-1}{2s}}).
\end{align*}

$\bullet$ {\it Case $\nu = 2$},
\begin{align*}
|v_k(t+\tau, x+h)-v_k(t, x)|\leq & \tau \sup_{\substack{t'\in(t+\tau, 0)\\x'\in B_{|x|+|h|}}} |\de_t v_k(t',x')| + h \sup_{\substack{t'\in(t+\tau, 0)\\x'\in B_{|x|+|h|}}} |\nabla_x v_k(t',x')|
\end{align*}
Now, we use that
\begin{align*}
\sup_{\substack{t'\in(t+\tau, 0)\\x'\in B_{|x|+|h|}}} |\nabla_x v_k(t',x')|\leq & \sup_{\substack{t'\in(t+\tau, 0)\\x'\in B_{|x|+|h|}}} |\nabla_x v_k(t',x')-\nabla_x v_k(0,x')|+\sup_{\substack{t'\in(t+\tau, 0)\\x'\in B_{|x|+|h|}}} |\nabla_x v_k(0,x')| \\
 \leq & |t|^{\frac{2s+\alpha-1}{2s}} + h \sup_{x'\in B_{|x|+|h|}}\left(|x'|\left[\sup_{x''\in B_{|x'|}} |D^2_x v_k(0,x'')|\right]\right)\\
 \leq & C\left(|t|^{\frac{2s+\alpha-1}{2s}} + |x|^{2s+\alpha-1}\right),
\end{align*}
so that in all we have
\[
|v_k(t+\tau, x+h)-v_k(t, x)|\leq C(|x|^\alpha+|x|^{2s+\alpha-1} + |t|^\beta + |t|^{\frac{2s+\alpha-1}{2s}}).
\]

In all three cases we deduce that, since $\alpha < 2s$ and $\alpha < 1$,
\begin{equation}
\label{eq.eps}
|v_k(t+\tau, x+h)-v_k(t, x)|\leq C\left(|x|^{2s-\epsilon} + |t|^{\frac{2s-\epsilon}{2s}}\right),
\end{equation}
where $\epsilon = \min\{2s-\alpha,1-\alpha\}> 0$ and $C$ is independent of $k$. We recall that the previous bound is found for $t\in\left(-\frac{1}{2}\rho_k^{-2s}-\tau,0\right]$ and $x\in \R^n$.

On the other hand, from the compactness of probability measures on the sphere we can find a subsequence of $\{L_k\}$ converging weakly to an operator $\tilde{L}$, that is, a subsequence of spectral measures $\{\mu_k\}$ converging to a spectral measure $\mu$ of an operator $\tilde{L}$ of the form \eqref{eq.L1}-\eqref{eq.L2}. Therefore, we have the ingredients to apply Lemma \ref{lem.1} to the sequence $v_k(t+\tau,x+h)-v_k(t, x)$.

Fixed any bounded sets $I\subset(-\infty,0]$, $K\subset \R^n$ we have
\begin{itemize}
\item{$v_k(t+\tau,x+h)-v_k(t, x)$ converges to $v(t+\tau,x+h)-v(t, x)$ uniformly over compact sets,}
\item{$(\de_t-L_k)(v_k(t+\tau,x+h)-v_k(t,x))$ converges uniformly on $K$ to the constant function $0$,}
\item{$|v_k(t+\tau, x+h)-v_k(t, x)| \leq C \left(1+|x|^{2s-\epsilon}\right)$ for all $x\in \R^n$, where $C$ now depends on $I$, which is fixed, and for $k$ large enough.}
\end{itemize}

Therefore, by Lemma \ref{lem.1} we deduce that
\[
(\de_t-\tilde{L})(v(t+\tau, x+h)-v(t, x)) = 0 \textrm{ in } I\times K.
\]
Since this can be done for any $I\subset (-\infty,0]$ and any $K\subset \R^n$, then
\[
(\de_t-\tilde{L})(v(t+\tau, x+h)-v(t, x)) = 0 \textrm{ in } (-\infty,0)\times \R^n.
\]
\\
%**********************STEP 4********************
{\bf Step 4: Contradiction.} Now, from the expression \eqref{eq.eps} and using the Liouville-type theorem in the entire space, Theorem \ref{thm.liouv}, we obtain that $v(t+\tau, x+h)-v(t, x)$ must be a polynomial in $x$ of degree at most $\lfloor\max\{\alpha,2s+\alpha-1\} \rfloor = \max\{0,\nu-1\}$.

This means that $v(t, x)$ is a polynomial in $x$ plus a polynomial in $t$, satisfying \eqref{eq.viszero}. Therefore, $v\equiv 0$, which is a contradiction with the expression \eqref{eq.nondeg} in the limit.
\end{proof}
%If $\nu = 0$ this means that $v(t+\tau, x+h)-v(t, x)$ must be constant, so that $v(t,x)$ is a polynomial of degree at most 1 in $x$ plus a polynomial of degree at most 1 in %$t$. However, since $v(0,0)= v_t(0,0) = 0$ and $[v(t,\cdot)]_{C^{2s+\alpha}_x(\R^n)} \leq 1$, we get that $v$ must be constant an equal to $0$, which is a contradiction by \eqref{eq.nondeg}.

%On the other hand, if $\nu = 1,2$ we deduce $v(t, x)$ is a polynomial in $x$ of degree at most $\nu$ plus a polynomial in $t$ of degree 1, but $v(0,0) = \dots = D^\nu_x v(0,0) = \de_t v (0,0) = 0$, so that $v \equiv 0$, contradicting again that $v$ is non-constant.

With the previous result we have the key ingredients to prove our main interior regularity estimate.

\begin{proof}[Proof of Theorem \ref{thm.maina}]

Let $\beta =\frac{\alpha}{2s} \in (0,1)$ as before.

Pick $\eta\in C_c^\infty(B_2)$ a cutoff function depending only on $x$ such that $\eta \equiv 1$ in $B_{3/2}$, and consider $w \in C^\infty_c((-\infty,0]\times\R^n)$, satisfying $\de_t w-Lw = f \textrm{ in } B_1$. Applying Proposition \ref{prop.main} to the function $\eta w$ we obtain that, for any $\delta$, there is a $C = C(\delta, n,s,\alpha,\lambda, \Lambda)$ such that
\begin{align*}
[w &]^{(1+\beta,2s+\alpha)}_{((-2^{-2s},0)\times B_{1/2})} \leq \delta [\eta w]^{(1+\beta,2s+\alpha)}_{((-1,0)\times B_2)} +\\
&+ C \left(\|w\|_{C^{1,\nu}_{t,x}((-1,0)\times B_1)} + [f]_{C^{\beta,\alpha}_{t,x}((-1,0)\times B_1)} + [(\de_t - L)(\eta w - w)]_{C^{\beta,\alpha}_{t,x}((-1,0)\times B_1)}\right).
\end{align*}

Now, since $\eta w - w$ vanishes in $B_{3/2}$ we have that
\begin{equation}
\label{eq.boundLphi}
[(\de_t - L)(\eta w - w)]_{C^{\beta,\alpha}_{t,x}((-1,0)\times B_1)} =[L(\eta w - w)]_{C^{\beta,\alpha}_{t,x}((-1,0)\times B_1)} \leq C\|w\|_{C^{\beta,\alpha}_{t,x}((-1,0)\times \R^n)}.
\end{equation}
Indeed, if we denote $\phi := \eta w - w$, we clearly have
\[
[\phi]_{C^{\beta,\alpha}_{t,x}((-1,0)\times\R^n)} \leq C \|w\|_{C^{\beta,\alpha}_{t,x}((-1,0)\times\R^n)}
\]
(for example using \eqref{eq.1.1} inside $B_2$ and noticing that $|\phi| = |w|$ outside $B_2$). Thus,
\begin{align*}
|L\phi(t, x)-L\phi(t',x')| & \leq C\int_{S^{n-1}} \int_{1/2}^\infty |-\phi(t, x+r\theta)+\phi(t',x'+r\theta)|\frac{dr}{|r|^{1+2s}}d\mu(\theta) \\
& \leq C[\phi]_{C^{\beta,\alpha}_{t,x}((-1,0)\times \R^n)} \left(|t-t'|^\beta+|x-x'|^\alpha\right)\int_{S^{n-1}} \int_{1/2}^\infty \frac{dr}{|r|^{1+2s}}d\mu(\theta),
\end{align*}
so that we reach our conclusion, since
\[
\int_{S^{n-1}} \int_{1/2}^\infty \frac{dr}{|r|^{1+2s}}d\mu(\theta) \leq \left(\int_{S^{n-1}}d\mu(\theta)\right)\int_{1/2}^\infty r^{-1-2s} dr \leq C \Lambda.
\]

The previous inequality, \eqref{eq.boundLphi}, combined with an inequality of the form \eqref{eq.1} for $[\eta w]^{(1+\beta,2s+\alpha)}_{((-1,0)\times B_2)}$, yields that for any $\delta > 0$, there exists a $C  = C(\delta, n,s,\alpha,\lambda, \Lambda)$ such that
\begin{align}
\label{eq.2}
[w]^{(1+\beta,2s+\alpha)}_{((-2^{-2s},0)\times B_{1/2})} &\leq \delta [w]^{(1+\beta,2s+\alpha)}_{((-1,0)\times B_2)} +\\
\nonumber & + C \left(\|w\|_{C^{1,\nu}_{t,x}((-1,0)\times B_1)} + [f]_{C^{\beta,\alpha}_{t,x}((-1,0)\times B_1)} + \|w\|_{C^{\beta,\alpha}_{t,x}((-1,0)\times \R^n)}\right).
\end{align}

Now, using interpolation \eqref{eq.interp}, for any $\kappa > 0$, there exists $\overline{C} = \overline{C}(\kappa, n, \alpha,\beta, s)$ such that
\[
\|w\|_{C^{1,\nu}_{t,x}((-1,0)\times B_1)}\leq \kappa [w]^{(1+\beta,2s+\alpha)}_{((-1,0)\times B_1)} + \overline{C}\|w\|_{L^\infty((-1,0)\times B_1)}.
\]
Fixing $\kappa= \delta/C$ with $C$ as in \eqref{eq.2}, we get
\[
[w]^{(1+\beta,2s+\alpha)}_{((-2^{-2s},0)\times B_{1/2})} \leq 2\delta [w]^{(1+\beta,2s+\alpha)}_{((-1,0)\times B_2)} + C \left([f]_{C^{\beta,\alpha}_{t,x}((-1,0)\times B_1)} + \|w\|_{C^{\beta,\alpha}_{t,x}((-1,0)\times \R^n)}\right).
\]
By a standard argument (see for example the Lemma after \cite[Theorem~2]{S} or the proof of \cite[Theorem 1.1 (b)]{RS}) we obtain that there exists a constant $C = C(n,s,\alpha,\lambda,\Lambda)$ such that
\[
[w]^{(1+\beta,2s+\alpha)}_{((-2^{-2s},0)\times B_{1/2})} \leq  C \left([f]_{C^{\beta,\alpha}_{t,x}((-1,0)\times B_1)} + \|w\|_{C^{\beta,\alpha}_{t,x}((-1,0)\times \R^n)}\right).
\]
Using interpolation again, it follows
\[
\|w\|^{(1+\beta,2s+\alpha)}_{((-2^{-2s},0)\times B_{1/2})} \leq  C \left([f]_{C^{\beta,\alpha}_{t,x}((-1,0)\times B_1)} + \|w\|_{C^{\beta,\alpha}_{t,x}((-1,0)\times \R^n)}\right).
\]
In particular, we obtain, for $w\in C_c^\infty((-\infty,0]\times\R^n)$
\begin{align*}
\sup_{x\in B_{1/2}} \|w\|_{C^{1+\beta}_t (-2^{-2s},0)} + \sup_{t\in(-2^{-2s},0)} & \|w\|_{C^{2s+\alpha}_x(B_{1/2})}\leq \\
&\leq  C \left([f]_{C^{\beta,\alpha}_{t,x}((-1,0)\times B_1)} + \|w\|_{C^{\beta,\alpha}_{t,x}((-1,0)\times \R^n)}\right).
\end{align*}
By a covering argument, the domain of $t$ on the left hand side can be easily replaced by $(-1/2, 0)$.

To get the result for general $u\in C^{\frac{\alpha}{2s},\alpha}_{t,x}((-1,0)\times\R^n)$ we can use a standard approximation argument. Indeed, if $u\in C^{\frac{\alpha}{2s},\alpha}_{t,x}$, and $\eta_\epsilon$ is a standard mollifier, then we regularise $u$ and notice that $(\de_t-L)(u*\eta_\epsilon) = f*\eta_\epsilon$. We now apply the result for smooth functions to $u * \eta_\epsilon$ and $f*\eta_\epsilon$, and take the limit as $\epsilon \downarrow 0$, to get the desired result.
\end{proof}

\subsection{Proof of Theorem \ref{thm.mainb}}
We next prove the interior regularity for $f\in L^\infty$, Theorem~\ref{thm.mainb}. To do so, we begin as in the previous case, with the following proposition, analogous to Proposition \ref{prop.main}.

\begin{prop}
\label{prop.main2}
Let $s\in(0,1)$, $\nu = \lceil 2s \rceil -1$, and let $L$ be an operator of the form \eqref{eq.L1}-\eqref{eq.L2}. Let $\epsilon > 0$ be such that $2s-\epsilon > \nu$ Assume $u\in C^\infty_c((-\infty,0]\times\R^n)$ satisfies
\[
\de_t u-Lu = f \textrm{ in } (-1,0)\times B_1
\]
with $f\in L^\infty((-1,0)\times B_1)$. Then, for any $\delta > 0$ we have
\begin{equation}
\label{eq.ineq1_2}
[u]^{(1-\frac{\epsilon}{2s},2s-\epsilon)}_{((-2^{-2s},0)\times B_{1/2})} \leq \delta [u]^{(1-\frac{\epsilon}{2s},2s-\epsilon)}_{((-1,0)\times \R^n)} + C \left(\|u\|_{C^{0,\nu}_{t,x} ((-1,0)\times B_1)} + \|f\|_{L^\infty((-1,0)\times B_1)} \right),
\end{equation}
where the constant $C$ depends only on $\delta$, $n$, $s$, $\epsilon$ and the ellipticity constants \eqref{eq.L2}.
\end{prop}
\begin{proof}
We follow the steps of Proposition \ref{prop.main}.

Suppose that for a given $\delta > 0$ the estimate does not hold for any constant $C$: for each $k \in \N$ there exist functions $w_k\in C^\infty_c((-\infty,0]\times\R^n), f_k\in L^\infty((-1,0)\times B_1)$, and operators $L_k$ of the form \eqref{eq.L1}-\eqref{eq.L2} such that $\de_t w_k-L_k w_k = f_k$ in $(-1,0)\times B_1$ and
\begin{equation}
\label{eq.ineqk_2}
[w_k]^{(1-\frac{\epsilon}{2s},2s-\epsilon)}_{((-2^{-2s},0)\times B_{1/2})} > \delta [w_k]^{(1-\frac{\epsilon}{2s},2s-\epsilon)}_{((-1,0)\times \R^n)} + k \left(\|w_k\|_{C^{0,\nu}_{t,x} ((-1,0)\times B_1)} + \|f_k\|_{L^\infty((-1,0)\times B_1)} \right).
\end{equation}

%****************STEP 1************************

{\bf Step 1: The blow-up parameter, $\rho_k$.} We only need to separate two cases according to the value of $\nu$ now.

$\bullet$ {\it Case $\nu = 0$.} By definition, we can choose $x_k,y_k\in B_{1/2},~t_k,s_k\in (-2^{-2s},0)$ such that
\begin{equation}
\label{eq.ineqk2.0_2}
\frac{1}{4}[w_k]^{(1-\frac{\epsilon}{2s},2s-\epsilon)}_{((-2^{-2s},0)\times B_{1/2})} < \frac{| w_k(t_k,x_k)- w_k(s_k,y_k)|}{|s_k-t_k|^{1-\frac{\epsilon}{2s}}+|x_k-y_k|^{2s-\epsilon}}.
\end{equation}

Define,
\[
\rho_k := |t_k-s_k|^{\frac{1}{2s}}+|x_k-y_k|.
\]
As in the proof of Proposition \ref{prop.main} there exists some small constant $\chi > 0$ depending only on $n$ such that
\begin{equation}
\label{eq.ineqk2.0.1_2}
\chi(n)[w_k]^{(1-\frac{\epsilon}{2s},2s-\epsilon)}_{((-2^{-2s},0)\times B_{1/2})} < \frac{| w_k(t_k,x_k)- w_k(s_k,y_k)|}{\rho_k^{2s-\epsilon}}.
\end{equation}

Therefore, we have
\begin{align*}
\chi(n)[w_k]^{(1-\frac{\epsilon}{2s},2s-\epsilon)}_{((-2^{-2s},0)\times B_{1/2})}
& < \frac{2\| w_k\|_{L^\infty((-2^{-2s},0)\times B_{1/2})}}{\rho_k^{2s-\epsilon}}\\
& \leq \frac{[w_k]^{(1-\frac{\epsilon}{2s},2s-\epsilon)}_{((-2^{-2s},0)\times B_{1/2})}}{\rho_k^{2s-\epsilon}k} ,
\end{align*}
where in the last inequality we are using \eqref{eq.ineqk_2}. Thus, we finally obtain that $\rho_k\to 0$ as $k\to \infty$.

$\bullet$ {\it Case $\nu = 1$.} Proceed as before, by choosing $x_k,y_k\in B_{1/2},~t_k,s_k\in (-2^{-2s},0)$ such that
\begin{equation}
\label{eq.ineqk2_2}
\frac{1}{2}[w_k]^{(1-\frac{\epsilon}{2s},2s-\epsilon)}_{((-2^{-2s},0)\times B_{1/2})} < \frac{|w_k(t_k,x_k)- w_k(s_k,x_k)|}{|t_k-s_k|^{1-\frac{\epsilon}{2s}}}+\frac{|\nabla_x w_k(t_k,x_k) - \nabla_x w_k(s_k,y_k)|}{|t_k-s_k|^{\frac{2s-\epsilon-1}{2s}}+|x_k-y_k|^{2s-\epsilon-1}}.
\end{equation}

Define $\rho_k$ as before, and up to a possible new choice of $y_k$ (as in the proof of the case $\nu = 0$ in the first step of Proposition \ref{prop.main}), we obtain that
\begin{equation}
\label{eq.ineqk2.0.2_2}
\chi(n)[w_k]^{(1-\frac{\epsilon}{2s},2s-\epsilon)}_{((-2^{-2s},0)\times B_{1/2})} < \frac{| w_k(t_k,x_k)- w_k(s_k,x_k)|}{\rho_k^{2s-\epsilon}}+\frac{| \nabla_x w_k(t_k,x_k) -  \nabla_x w_k(s_k,y_k)|}{\rho_k^{2s-\epsilon-1}}.
\end{equation}
A reasoning similar to the one before yields $\rho_k \to 0$ as $k\to \infty$ again.
\\[0.5cm]
%****************STEP 2*****************
{\bf Step 2: The blow-up sequence.} We begin by defining the following functions, where we will assume that $t_k \geq s_k$ (otherwise, we can swap them),
\[
v_k(t,x):= \frac{w_k(t_k+\rho_k^{2s} t,x_k+\rho_k x) - p_k(x)}{\rho_k^{2s-\epsilon} [w_k]^{(1-\frac{\epsilon}{2s},2s-\epsilon)}_{((-1,0)\times \R^n)}}.
\]
Here $p_k(x)$ is a polynomial in $x$ of degree at most $\nu$, such that
\begin{equation}
\label{eq.vkzero2}
v_k(0,0) = D^\nu_x v_k(0,0) = 0.
\end{equation}

Thanks to the scaling of the seminorm (see \eqref{eq.rescnorm_2}), $v_k$ satisfies
\begin{equation}
\label{eq.bdsemi_2}
[v_k]^{(1-\frac{\epsilon}{2s},2s-\epsilon)}_{\left(\left(-\frac{1}{2}\rho_k^{-2s},0\right]\times \R^n \right)} \leq 1.
\end{equation}

We also have uniform convergence towards $0$ of the following quantity for fixed $\tau \in (-1,0)$ and $h\in B_1$,
\begin{equation}
\label{eq.bdk_2}
|(\de_t-L_k)(v_k(t+\tau,x+h)-v_k(t, x))| \leq \frac{2\rho_k^{\epsilon}}{k} \to 0
\end{equation}
uniformly in $\left(-\frac{1}{2}\rho_k^{-2s},0\right)\times B_{\left(\frac{1}{2}\rho_k^{-1} - h\right)}$. Indeed,
\begin{align*}
 |(\de_t- & L_k)(v_k(t+\tau,x+h)-v_k(t, x))| = \\
& = \frac{\rho_k^{\epsilon}}{ [w_k]^{(1-\frac{\epsilon}{2s},2s-\epsilon)}_{((-1,0)\times \R^n)}}\cdot |f_k(t_k+\rho_k^{2s} (t+\tau),x_k+\rho_k (x+h)) - f_k(t_k+\rho_k^{2s} t,x_k+\rho_k x)| \\
&\leq 2\frac{\rho_k^{\epsilon}\|f_k\|_{L^\infty((-1,0)\times B_1)}}{[w_k]^{(1-\frac{\epsilon}{2s},2s-\epsilon)}_{((-1,0)\times \R^n)}}\leq \frac{2\rho_k^{\epsilon}}{k} \to 0,
\end{align*}
where in the last inequality we have used \eqref{eq.ineqk_2}.

We now define the following points in the set $[-1,0]\times\overline{B_1}$,
\[
\xi_k = \left( \frac{s_k-t_k}{\rho_k^{2s}},\frac{y_k-x_k}{\rho_k} \right),~~~~
\xi_k^{(1)} = \left( \frac{s_k-t_k}{\rho_k^{2s}},0 \right),
\]
and notice that we have
\begin{alignat*}{2}
 v_k(\xi_k)= &\frac{w_k(s_k,y_k) - w_k(t_k,x_k)}{\rho_k^{2s-\epsilon} [w_k]^{(1-\frac{\epsilon}{2s},2s-\epsilon)}_{((-1,0)\times \R^n)}}&&~~~\textrm{ for }\nu = 0\\[4mm]
 v_k(\xi_k^{(1)})= &\frac{ w_k(s_k,x_k) -  w_k(t_k,x_k)}{\rho_k^{2s-\epsilon} [w_k]^{(1-\frac{\epsilon}{2s},2s-\epsilon)}_{((-1,0)\times \R^n)}}&& ~~~\textrm{ for } \nu = 1,
\\[4mm]
\nabla_x v_k(\xi_k)= &\frac{\nabla_x w_k(s_k,y_k) - \nabla_x w_k(t_k,x_k)}{\rho_k^{2s-\epsilon-1} [w_k]^{(1-\frac{\epsilon}{2s},2s-\epsilon)}_{((-1,0)\times \R^n)}}&& ~~~\textrm{ for } \nu = 1.
\end{alignat*}
Hence, combining \eqref{eq.ineqk2.0.1_2}-\eqref{eq.ineqk2.0.2_2} with \eqref{eq.ineqk_2} we obtain
\begin{alignat}{2}
\label{eq.nondeg_2}
|v_k(\xi_k)| >& \chi(n) \delta && ~~~~\textrm{ if }\nu = 0,\nonumber \\[4mm]
 |v_k(\xi_k^{(1)})|+|\nabla_x v_k(\xi_k)| >& \chi(n) \delta && ~~~~ \textrm{ if } \nu = 1.
\end{alignat}

Notice that, up to a subsequence, $\xi_k$ converge to some $\xi\in [0,1]\times\overline{B_1}$ (and so do $\xi_k^{(1)}$) so that from now on we will restrict ourselves to this subsequence.
\\[0.5cm]
%**************************STEP 3**************************
{\bf Step 3. Convergence properties of the blow-up sequence.} Recall that we have uniform bound on the seminorms of $v_k$, \eqref{eq.bdsemi_2}, we deduce that, up to subsequences, $v_k$ converges in $C^\epsilon_t$ and in $C^{\nu+\epsilon}_x$ to some function $v$ over compact subsets of $(-\infty,0]\times\R^n$. This follows since the Hölder seminorms $[v_k]_{C^{1-\frac{\epsilon}{2s}}_{t}\left(\left(-\frac{1}{2}\rho_k^{-2s},0\right]\times\R^n\right)}$, $[ v_k]_{C^{2s-\epsilon}_x\left(\left(-\frac{1}{2}\rho_k^{-2s},0\right]\times\R^n\right)}$ are uniformly bounded with respect to $k\in \N$, and the domains are expanding to $(-\infty,0]\times\R^n$.

We restrict ourselves to this subsequence, and obtain a limit function $v$ defined in $(-\infty,0]\times\R^n$ such that
\begin{equation}
\label{eq.viszero_2}
v(0,0) = D^\nu_x v(0,0) = 0~~~\textrm{and}~~~
[v]^{(1-\frac{\epsilon}{2s},2s-\epsilon)}_{\left(\left(-\infty,0\right]\times \R^n \right)} \leq 1.
\end{equation}
By \eqref{eq.nondeg_2} and the nice convergence, we get that $v$ cannot be constant.

Now consider the functions $v_k(t+\tau, x+h)-v_k(t, x)$ for fixed $\tau \in (-1,0), h\in B_1$. We want to compute an upper bound for $|v_k(t+\tau, x+h)-v_k(t, x)|$  depending on $t$ and $x$, such that $t\in\left(-\frac{1}{2}\rho_k^{-2s}-\tau,0\right]$, $x\in \R^n$, and we separate two cases:

$\bullet$ {\it Case $\nu = 0$},
\[
|v_k(t+\tau, x+h)-v_k(t, x)| \leq C
\]
using the bounds on the seminorm of $v_k$ and \eqref{eq.vkzero2}, and where $C$ can depend on $\tau$ and $h$.

$\bullet$ {\it Case $\nu = 1$},
\begin{align*}
|v_k(t+\tau, x+h)-v_k(t, x)|\leq C\left(|x|^{2s-\epsilon-1}+|t|^{\frac{2s-\epsilon-1}{2s}}+1\right).
\end{align*}

As before we can assume that $L_k$ converges to $\tilde{L}$ and then, using Lemma~\ref{lem.1}, we find
\[
(\de_t-\tilde{L})(v(t+\tau, x+h)-v(t, x)) = 0 \textrm{ in } (-\infty,0)\times \R^n.
\]
\\
%**********************STEP 4********************
{\bf Step 4: Contradiction.} From the Liouville-type theorem in the entire space, Theorem \ref{thm.liouv}, we obtain that $v(t+\tau, x+h)-v(t, x)$ must be constant, and therefore $v(t,x)$ is a polynomial of degree at most~1 in $x$ plus a polynomial of degree at most~1 in $t$. Therefore, by \eqref{eq.viszero_2} we get $v\equiv 0$, which is a contradiction with the expression \eqref{eq.nondeg_2} in the limit.
\end{proof}

Using the previous proposition we can prove Theorem \ref{thm.mainb}.

\begin{proof}[Proof of Theorem \ref{thm.mainb}]
Pick $\eta\in C_c^\infty(B_2)$ a cutoff function depending only on $x$ such that $\eta \equiv 1$ in $B_{3/2}$ and $0\leq \phi \leq 1 \textrm{ in } B_2$, and consider $w \in C^\infty_c(-\infty,0]\times\R^n$, satisfying $\de_t w-Lw = f \textrm{ in } B_1$. Applying Proposition \ref{prop.main2} to the function $\eta w$ we obtain that, for any $\delta$, there is a $C = C(\delta, n,s,\epsilon,\lambda, \Lambda)$ such that
\begin{align}
\nonumber [w &]^{(1-\frac{\epsilon}{2s},2s-\epsilon)}_{((-2^{-2s},0)\times B_{1/2})} \leq \delta [\eta w]^{(1-\frac{\epsilon}{2s},2s-\epsilon)}_{((-1,0)\times B_2)}+ \\
\label{eq.thm13} &+ C \left(\|w\|_{C^{0,\nu}_{t,x}((-1,0)\times B_1)} + \|f\|_{L^\infty((-1,0)\times B_1)} + \|(\de_t - L)(\eta w - w)\|_{L^\infty((-1,0)\times B_1)}\right).
\end{align}
Now, since $\eta w - w$ vanishes in $B_{3/2}$ we have that
\begin{equation}
\label{eq.boundLphi_2}
\|(\de_t - L)(\eta w - w)\|_{L^\infty((-1,0)\times B_1)} =\|L(\eta w - w)\|_{L^\infty((-1,0)\times B_1)} \leq C\|w\|_{L^\infty((-1,0)\times \R^n)}.
\end{equation}
Indeed, if we denote $\phi := \eta w - w$, we clearly have
\[
\|\phi\|_{L^\infty((-1,0)\times\R^n)} \leq \|w\|_{L^\infty((-1,0)\times\R^n)},
\]
and now
\begin{align*}
|L\phi(t, x)| & \leq C\int_{S^{n-1}}\int_{1/2}^\infty |-\phi(t, x+r\theta)|\frac{dr}{|r|^{1+2s}}d\mu(\theta) \\
& \leq C\|\phi\|_{L^\infty((-1,0)\times \R^n)} \int_{S^{n-1}}\int_{1/2}^\infty \frac{dr}{|r|^{1+2s}}d\mu(\theta) \leq C\Lambda \|\phi\|_{L^\infty((-1,0)\times \R^n)},
\end{align*}
where $C$ depends only on $n$ and $s$.

The previous inequality, \eqref{eq.boundLphi_2}, together with $\|\eta w\|_{C^{0,\nu}_{t,x}(-1,0)\times B_2}\leq C\|w\|_{C^{0,\nu}_{t,x}(-1,0)\times B_2}$ (for $C$ depending on $\eta$ fixed) yields that for any $\delta > 0$, there exists a constant $C  = C(\delta, n,s,\epsilon,\lambda, \Lambda)$ such that
\begin{align}
\label{eq.2_2}
[w &]^{(1-\frac{\epsilon}{2s},2s-\epsilon)}_{((-2^{-2s},0)\times B_{1/2})}\leq \\
\nonumber & \leq \delta [w]^{(1-\frac{\epsilon}{2s},2s-\epsilon)}_{((-1,0)\times B_2)} + C \left(\|w\|_{C^{0,\nu}_{t,x}((-1,0)\times B_1)} + \|f\|_{L^\infty((-1,0)\times B_1)} + \|w\|_{L^\infty((-1,0)\times \R^n)}\right).
\end{align}
Using interpolation \eqref{eq.interp2}, for any $\kappa > 0$, there exists $\overline{C} = \overline{C}(\kappa, n, s, \epsilon)$ such that
\[
\|w\|_{C^{0,\nu}_{t,x}((-1,0)\times B_1)}\leq \kappa [w]^{(1-\frac{\epsilon}{2s},2s-\epsilon)}_{((-1,0)\times B_1)} + \overline{C}\|w\|_{L^\infty((-1,0)\times B_1)}.
\]
Fixing $\kappa = \delta/C$ with $C$ as in \eqref{eq.2_2}, we get
\[
[w]^{(1-\frac{\epsilon}{2s},2s-\epsilon)}_{((-2^{-2s},0)\times B_{1/2})} \leq 2\delta [w]^{(1-\frac{\epsilon}{2s},2s-\epsilon)}_{((-1,0)\times B_2)} + C \left(\|f\|_{L^\infty((-1,0)\times B_1)} + \|w\|_{L^\infty((-1,0)\times \R^n)}\right).
\]

From which, as in the proof of Theorem \ref{thm.maina}, there exists a constant $C = C(n,s,\epsilon,\lambda,\Lambda)$ such that
\[
\|w\|^{(1-\frac{\epsilon}{2s},2s-\epsilon)}_{((-2^{-2s},0)\times B_{1/2})} \leq  C \left(\|f\|_{L^\infty_{t,x}((-1,0)\times B_1)} + \|w\|_{L^\infty((-1,0)\times \R^n)}\right),
\]
and this implies the bound we wanted.

To get the result for general $u\in L^\infty((-1,0)\times\R^n)$ we use a standard approximation argument, and we are done.
\end{proof}

Let us now give a corollary on the regularity of solutions without any constrain in the relation between $\alpha$ and $\beta$.

\begin{cor}
\label{cor.main1}
Let $s\in(0,1)$, and let $L$ be any operator of the form \eqref{eq.L1}-\eqref{eq.L2}. Let $u$ be any bounded weak solution to \eqref{eq.frac}. Let
\[
C_0 = \|u\|_{C^{\beta,\alpha}_{t,x}((-1,0)\times \R^n)} + \|f\|_{C^{\beta,\alpha}_{t,x}((-1,0)\times B_1)}.
\]
Then,
\begin{equation}
\label{eq.mainineq_cor}
\|u\|_{C^{1+\beta-\epsilon}_t \left(\left(-\frac{1}{2},0\right) \times B_{1/2}\right)} + \|u\|_{C^{2s+\alpha-\epsilon}_x\left(\left(-\frac{1}{2},0\right) \times B_{1/2}\right)} \leq C C_0,
\end{equation}
for any $\epsilon > 0$, where the constant $C$ depends only on $\epsilon, n, s$ and the ellipticity constants \eqref{eq.L2}.
\end{cor}
\begin{proof}

Define the following incremental quotients in $x$,
\[
u^h_{\alpha}(t, x) := \frac{u(t, x+h)-u(t, x)}{|h|^{\alpha}},~~~~f^h_{\alpha}(t, x) := \frac{f(t, x+h)-f(t, x)}{|h|^{\alpha}}
\]
for some $h\in\R^n$ fixed. Notice that
\[
\de_t u^h_{\alpha} - L u^h_{\alpha} = f^h_{\alpha} \textrm{ in } (-1,0)\times B_{1-|h|}.
\]
We apply Theorem \ref{thm.mainb} to the previous functions, reaching
\begin{align*}
\sup_{x\in B_{1/2}} \|u^h_{\alpha}\|_{C^{1-\epsilon}_t (-2^{-2s},0)} &+ \sup_{t\in(-2^{-2s},0)} \|u^h_{\alpha}\|_{C^{2s-\epsilon}_x(B_{1/2})} \leq  \\
& \leq  C \left(\|f^h_{\alpha}\|_{L^\infty((-1,0)\times B_{1-|h|})} + \|u^h_{\alpha}\|_{L^\infty((-1,0)\times \R^n)}\right).
\end{align*}

Now, since
\begin{align*}
\|f^h_{\alpha}\|_{L^\infty((-1,0)\times B_{1-|h|})}  & \leq C\left( [f]_{C^{\alpha}_x ((-1,0)\times B_{1} ) } \right),\\
\|u^h_{\alpha}\|_{L^\infty((-1,0)\times \R^n)} & \leq C\left( [u]_{C^{\alpha}_x ((-1,0)\times \R^n ) } \right),
\end{align*}
and
\[
 \sup_{h\in B_{1/4}} \sup_{t\in(-2^{-2s},0)} \|u^h_{\alpha}\|_{C^{2s-\epsilon}_x(B_{1/2})} \geq   \sup_{t\in(-2^{-2s},0)} \|u\|_{C^{2s-\epsilon+\alpha}_x(B_{1/2})},
\]
(see for example \cite[Lemma 5.6]{CC}) we obtain
\[
\sup_{t\in(-2^{-2s},0)} \|u\|_{C^{2s-\epsilon+\alpha}_x(B_{1/2})}\leq C\left( [f]_{C^{\alpha}_x ((-1,0)\times B_{1} ) }  +[u]_{C^{\alpha}_x ((-1,0)\times \R^n ) } \right).
\]

The same can be done taking incremental quotients in $t$, and adding up both inequalities we reach the desired result.
\end{proof}

\begin{rem}
\label{rem.main1}
The previous corollary is still true if we only subtract an arbitrarily small $\epsilon$ to one of the terms in the left hand side of \eqref{eq.mainineq_cor}. For example, if $2s\beta < \alpha$ then we apply Theorem \ref{thm.maina} with indices $\beta$ and $\alpha' = 2s\beta$ and combine it with the argument from Corollary \ref{cor.main1}.
\end{rem}

When the nonlocal parabolic equation for the operator $L$ is fulfilled in the entire space $\R^n$ we have a nice result where we no longer require a priori spatial regularity of the solution.

\begin{cor}
\label{cor.main2}
Let $s\in(0,1)$, and let $L$ be any operator of the form \eqref{eq.L1}-\eqref{eq.L2}. Let $u$ be any bounded weak solution to
\begin{equation}
\label{eq.frac_Rn}
\de_t u-Lu = f \textrm{ in } (-1,0)\times \R^n.
\end{equation}

Let $\alpha\in (0,1)$ be such that $\frac{\alpha}{2s}\in (0,1)$, and
\[
C_0 = \|u\|_{L^\infty((-1,0)\times \R^n)} + \|f\|_{C^{\frac{\alpha}{2s},\alpha}_{t,x}((-1,0)\times \R^n)}.
\]
Then, if $\alpha + 2s$ is not an integer
\begin{equation*}
 \|u\|_{C^{1+\frac{\alpha}{2s}}_t \left(\left(-\frac{1}{2},0 \right)\times \R^n\right)} +  \|u\|_{C^{2s+\alpha}_x\left(\left( -\frac{1}{2},0 \right)\times \R^n\right)} \leq C C_0.
\end{equation*}
The constant $C$ depends only on $n, s, \alpha$ and the ellipticity constants \eqref{eq.L2}.
\end{cor}
\begin{proof}
Simply apply Theorem \ref{thm.maina} to balls covering $\R^n$ to get
\begin{align*}
 \|u\|_{C^{1+\frac{\alpha}{2s}}_t \left(\left(-\frac{1}{2},0 \right)\times \R^n\right)} & +  \|u\|_{C^{2s+\alpha}_x\left(\left( -\frac{1}{2},0 \right)\times \R^n\right)} \leq \\
 &\leq C \left(\|u\|_{C^{\frac{\alpha}{2s}}_{t}\left(\left(-\frac{3}{4},0\right)\times \R^n\right)} + \|f\|_{C^{\frac{\alpha}{2s},\alpha}_{t,x}\left(\left(-\frac{3}{4},0\right)\times \R^n\right)}\right).
\end{align*}
On the other hand, from Theorem \ref{thm.mainb} applied again to balls covering $\R^n$ we have
\begin{equation*}
\|u\|_{C^{\frac{\alpha}{2s}}_{t}\left(\left(-\frac{3}{4},0\right)\times \R^n\right)}  \leq C \left(\|u\|_{L^\infty \left(\left(-1,0\right)\times \R^n\right)} + \|f\|_{L^\infty\left(\left(-1,0\right)\times \R^n\right)}\right),
\end{equation*}
where we took $\epsilon = 2s-\alpha >0$. Combining both expressions we obtain the desired result.
\end{proof}

We next prove a result when the kernels have some regularity.

\begin{cor}
\label{cor.main3}
Let $s\in(0,1)$, and let $L$ be any operator of the form \eqref{eq.L3} with bounds \eqref{eq.L2}. Assume that
\[
a\in C^\alpha(S^{n-1}),
\]
for some $\alpha \in (0,1)$ such that $\alpha < 2s$.

Let $u$ be any bounded weak solution to \eqref{eq.frac}, and
\[
C_0 = \|u\|_{C^{\frac{\alpha}{2s}}_{t}((-1,0)\times \R^n)} + \|f\|_{C^{\frac{\alpha}{2s},\alpha}_{t,x}((-1,0)\times B_1)}.
\]

Then, if $2s+\alpha$ is not an integer,
\begin{equation}
\label{eq.mainineq_cor_2}
 \|u\|_{C^{1+\frac{\alpha}{2s}}_t \left(\left(-\frac{1}{2},0\right) \times B_{1/2}\right)} +  \|u\|_{C^{2s+\alpha}_x\left(\left(-\frac{1}{2},0\right) \times B_{1/2}\right)} \leq C C_0,
\end{equation}
for some constant $C$ depending only on $\alpha, n, s, \|a\|_{C^\alpha(S^{n-1})}$ and the ellipticity constants \eqref{eq.L2}.
\end{cor}

Notice that now on the right hand side of the estimate the term depending on $u$ no longer requires a $C^\alpha$ regularity in the $x$ variable. Instead, only uniform regularity in $\R^n$ in $t$ is required.

\begin{proof}
The proof reduces to see that, in the proof of Theorem \ref{thm.maina}, we can replace the bound \eqref{eq.boundLphi} (recall $\phi = \eta w - w$ for a cutoff function $\eta$) by
\[
[L \phi]_{C^{\frac{\alpha}{2s}, \alpha}_{t, x}((-1,0)\times B_1)}\leq C\|w\|_{C_t^{\frac{\alpha}{2s}} ((-1,0)\times \R^n)}.
\]
Indeed,
\begin{align*}
|L\phi(t, x)& -L \phi(t', x') | = \\
&= \left| \int_{\R^n\setminus B_{1/2}(x)} \phi(t, z)K(z-x)dz - \int_{\R^n\setminus B_{1/2}(x')} \phi(t', z)K(z-x)dz \right|\\
&\leq  C|t-t'|^{\frac{\alpha}{2s}} [\phi]_{C^{\frac{\alpha}{2s}}_t ((-1,0)\times\R^n)}+\\
 & ~~~~~~~~~~~+ \int_{\R^n\setminus B_{1/4}(x)} |\phi(t', z)|\cdot \left| K(z-x)-K(z-x')\right|dz ,
\end{align*}
where we have assumed without loss of generality that $B_{1/8}(x') \subset B_{1/4}(x) \subset B_{1/2}(x')$, by considering $|x-x'|<\frac{1}{8}$, and where
\[
K\left(y\right) = \frac{a(y/|y|)}{|y|^{n+2s}}.
\]

Notice that $K$ is $C^\alpha(B_2\setminus B_{1/8})$ by being quotient of $C^\alpha$ functions, and therefore
\[
|K(z-x)-K(z-x')| \leq C|x-x'| ~~ \textrm{ for } ~~ z-x,z-x' \in B_2\setminus B_{1/8},
\]
where $C$ depends only on $n$ and $\|a\|_{C^\alpha(S^{n-1})}$. By homogeneity of $K$, for $z\in \R^n\setminus B_{1/4}$,
\begin{align*}
|K(z-x)-K(z-x')| &=\frac{1}{|z-x|^{n+2s}}\left|K\left(\frac{z-x}{|z-x|}\right)-K\left(\frac{z-x'}{|z-x|}\right)\right|\\
& \leq C\frac{1}{|z-x|^{n+2s+\alpha}}|x-x'|^\alpha,
\end{align*}
where we used that $\frac{z-x'}{|z-x|}\in B_2\setminus B_{1/8}$.

In all we have that
\[
|L\phi(t, x) -L \phi(t', x') | \leq C\|\phi\|_{C^{\frac{\alpha}{2s}}((-1,0)\times\R^n)}\left(|t-t'|^{\frac{\alpha}{2s}}+|x-x'|^{\alpha}\right),
\]
as desired.
\end{proof}

Finally, let us combine some of the results that have been obtained here to show the following result: when the kernel of the operator is regular enough, we gain $2s-\epsilon$ spatial interior regularity. That is

\begin{cor}
\label{cor.main4}
Let $s\in(0,1)$, and let $L$ be any operator of the form \eqref{eq.L3} with bounds \eqref{eq.L2}. Assume that
\[
a\in C^{k+\alpha}(S^{n-1}),
\]
for $\alpha \in (0,1)$, $k \in \N$.

Let $u$ be any bounded weak solution to \eqref{eq.frac}. Then, if $2s+\alpha$ is not an integer,
\begin{equation}
\label{eq.mainineq_cor_3}
\|u\|_{C^{2s+k+\alpha-\epsilon}_x\left(\left(-\frac{1}{2},0\right) \times B_{1/2}\right)} \leq C \left(\|u\|_{L^\infty((-1,0)\times \R^n)} + \|f\|_{C^{k+\alpha}_{x}((-1,0)\times B_1)}\right),
\end{equation}
for all $\epsilon > 0$ and for some constant $C$ depending only on $\epsilon, \alpha, n, s, \|a\|_{C^{k+\alpha}(S^{n-1})}$ and the ellipticity constants \eqref{eq.L2}.
\end{cor}
\begin{proof}
Let $\eta = \eta(x)$ be a cutoff function supported in $B_2$ and such that $\eta \equiv 1$ in $B_{3/2}$. In the expression \eqref{eq.thm13} from the proof of Theorem~\ref{thm.mainb} we can take incremental quotients of order $k+\alpha$ as in the proof of Corollary~\ref{cor.main1} to find
\begin{align*}
 [& u ]_{C^{2s-\epsilon +k+\alpha}_x((-2^{-2s},0)\times B_{1/2})} \leq \delta [\eta u]_{C^{2s-\epsilon+k+\alpha}_x((-1,0)\times B_2)}+ \\
&+ C \left(\|u\|_{C^{\nu+k+\alpha}_{x}((-1,0)\times B_1)} + \|f\|_{C^{k+\alpha}_x((-1,0)\times B_1)} + \|(\de_t - L)(\eta u - u)\|_{C^{k+\alpha}_x((-1,0)\times B_1)}\right).
\end{align*}

Notice that, as in Corollary~\ref{cor.main3}, we obtain
\[
\|(\de_t - L)(\eta u - u)\|_{C^{k+\alpha}_x((-1,0)\times B_1)} = \|L(\eta u - u)\|_{C^{k+\alpha}_x((-1,0)\times B_1)} \leq \|u\|_{L^\infty((-1,0)\times\R^n)},
\]
and now the desired result follows as in the proof of Theorem~\ref{thm.mainb}.
\end{proof}

\section{$C^s$ regularity up to the boundary}
\label{sec.4}
In this section we start the study of the regularity up to the boundary. We will first construct a supersolution with appropriate behaviour near the boundary, and then we establish the $C^s_x$ regularity up to the boundary, Proposition~\ref{prop.mainboundary}.

\subsection{A supersolution} Let us begin with the following general result, which gives a fractional Sobolev inequality for operators $L$ of the form \eqref{eq.L1}-\eqref{eq.L2}.
\begin{lem}[A fractional Sobolev inequality]
\label{lem.sob}
Let $s\in(0,1)$, $2s < n$, and $L$ any operator of the form \eqref{eq.L1}-\eqref{eq.L2}, with spectral measure $\mu$. Then,
\begin{equation}
\label{eq.sobolev}
\|f\|^2_{L^{\frac{2n}{n-2s}}(\R^n)} \leq C \int_{\R^n} \int_{S^{n-1}}\int_{-\infty}^\infty |f(x+\theta r)-f(x)|^2 \frac{dr}{|r|^{1+2s}} d\mu(\theta) dx
\end{equation}
for some constant $C$ depending only on $n$, $s$ and the ellipticity constants \eqref{eq.L2}.
\end{lem}

\begin{proof}
This fractional Sobolev inequality is already known when the operator $L$ is the fractional Laplacian. In this case the right hand side is the Gagliardo seminorm $[f]_{H^s(\R^n)}$.

Let us call $[f]_{H^s_L(\R^n)}$ the right hand side of \eqref{eq.sobolev}. Using the classical fractional Sobolev inequality and Plancherel's theorem, it is enough to prove that $[f]_{H^s_L(\R^n)}$ and $[f]_{H^s(\R^n)}$ are equivalent seminorms in the Fourier side. This follows by noticing that the Fourier symbol $A(\xi)$ of $L$ can be explicitly written as
\[
A(\xi) = \int_{S^{n-1}} |\xi \cdot \theta|^{2s}d\mu(\theta),
\]
(see for example \cite{ST}), so that
\[
[f]_{H_L^s(\R^n)} = \int_{\R^n} A(\xi)|\hat f(\xi)|^2d\xi.
\]
Now, by definition of $\lambda, \Lambda$, the ellipticity constants in \eqref{eq.L2}, we have
\[
0 < \lambda |\xi|^{2s} \leq A(\xi)\leq \Lambda |\xi|^{2s}.
\]
Using that the Fourier symbol of the fractional Laplacian is $|\xi|^{2s}$ we are done.
\end{proof}

We now give a result regarding the eigenfunctions associated to an operator $L$ in a domain $\Omega$. This will be used later to construct a supersolution.

\begin{lem}
\label{lem.bdeig}
Let $L$ be an operator of the form \eqref{eq.L1}-\eqref{eq.L2}, and let $\Omega\subset \R^n$ be a bounded Lipschitz domain. Then the eigenfunctions of the Dirichlet elliptic problem are bounded in $\Omega$. That is, if $\phi_k\in L^2(\Omega)$ is the eigenfunction associated to the $k$-th eigenvalue $\lambda_k$,
\begin{equation}\label{eq.fracheat1}
  \left\{ \begin{array}{rcll}
  L \phi_k &=&-\lambda_k \phi_k& \textrm{in }\Omega,\\
  u&=&0& \textrm{in } \R^n\setminus \Omega,\\
  \end{array}\right.
\end{equation}
then,
\[
\|\phi_k\|_{L^\infty(\Omega)} \leq C_1 \lambda_k^{\frac{1}{2}\left(\frac{n}{2s}\right)^2} \|\phi_k\|_{L^2(\Omega)}
\]
where $C_1$ is a constant depending only on $n$ and $s$. Moreover,
\[
\lim_{k\to \infty} \lambda_k k^{-\frac{2s}{n}}= C_2,
\]
for some constant $C_2$ depending only on $n, s, \Omega$ and the ellipticity constants \eqref{eq.L2}.
\end{lem}
\begin{proof} Let us begin with the first inequality. If $n = 1$, then $L$ is a multiple of the fractional Laplacian and the result is already known (e.g. \cite[Proposition 4]{SV} or \cite{FR}).
If $n > 2s$, it can be proved doing exactly the same as Servadei and Valdinoci do in \cite[Proposition 4]{SV} for the fractional Laplacian case. To do so, we use the fractional Sobolev inequality for general stable operators established in Lemma~\ref{lem.sob}.

The second result follows from \cite[Proposition 2.1]{FR}, which is a direct consequence of a result from \cite{G}.
\end{proof}

In this section and what follows we will use the following notation
\begin{equation}
d(x):=\textrm{dist}(x,\R^n\setminus\Omega).
\end{equation}

\begin{lem}[Supersolution]
\label{lem.super}
Let $s\in(0,1)$ and let $\Omega\subset \R^n$ be any bounded $C^{1,1}$ domain. Let $u$ be the solution to
\begin{equation}\label{eq.fracheat2}
  \left\{ \begin{array}{rcll}
  \de_t u - L u&=&1& \textrm{in }\Omega,\ t > 0 \\
  u&=&0& \textrm{in } \R^n\setminus \Omega,\ t \geq 0 \\
  u(0,x)&=&1&\textrm{in }\Omega.
  \end{array}\right.
\end{equation}

Then, we have
\begin{equation}
\label{eq.ds}
|u| \leq C(t_0) d^s,~~\textrm{ for all } t\geq t_0 > 0,
\end{equation}
where $C(t_0)$ depends only on $t_0, n, s, \Omega$ and the ellipticity constants \eqref{eq.L2}. The dependence with respect to $\Omega$ is via $|\Omega|$ and the $C^{1,1}$ norm of the domain.
\end{lem}
\begin{rem}
We call the $C^{1,1}$ norm of the domain to the maximum $\rho$ such that there are balls tangent at every point from inside and outside the domain with radius $\rho$.
\end{rem}
\begin{proof}
Notice that $u(t, x) = u_1(x)+u_2(t,x)$ where $u_1$ solves
\begin{equation}
  \left\{ \begin{array}{rcll}
  - L u_1&=&1& \textrm{in }\Omega,\ t > 0 \\
  u_1&=&0& \textrm{in } \R^n\setminus \Omega,\ t \geq 0 \\
  \end{array}\right.
\end{equation}
and $u_2$ solves
\begin{equation}
  \left\{ \begin{array}{rcll}
  \de_t u_2 - L u_2&=&0& \textrm{in }\Omega,\ t > 0 \\
  u_2&=&0& \textrm{in } \R^n\setminus \Omega,\ t \geq 0 \\
  u_2(0,x)&=&1-u_1(x)&\textrm{in }\Omega.
  \end{array}\right.
\end{equation}

By the results in \cite{RS} we have a bound for $u_1$ of the form
\begin{equation}
\label{eq.ds2}
|u_1| \leq C d^s,
\end{equation}
where $C$ depends only on $n$, $s$, the $C^{1,1}$ norm of $\Omega$ and the ellipticity constants \eqref{eq.L2}.

To bound $u_2$ we proceed as in the proof of \cite[Theorem 1.1]{FR} by expressing $u_2$ with respect to the eigenfunctions of the elliptic problem. Namely,
\[
u_2(t, x) = \sum_{k > 0}u_k \phi_k e^{-\lambda_k t},
\]
where $\phi_k$ is the $k$-th eigenfunction corresponding to the $k$-th eigenvalue $\lambda_k$, and $u_k$ are the Fourier coefficients of $u_2(0, x)$. We are assuming $\|\phi_k\|_{L^2(\Omega)} = 1$ for all $k\in \N$.

By the results in \cite{RS} for the elliptic problem and Lemma \ref{lem.bdeig} we have
\[
|\phi_k(x)| \leq C\lambda\|\phi_k\|_{L^\infty(\Omega)}d^s(x) \leq C \lambda_k^{w+1}d^s(x),
\]
where $w = \frac{1}{2}\left(\frac{n}{2s}\right)^2$ and $C$ depends only on $n$ and $s$. Therefore,
\[
|u_2(t, x)| \leq d^s(x)\sum_{k> 0} u_k C\lambda_k^w e^{-\lambda_k t},~~~\textrm{ for all } x \in \Omega.
\]
Using exactly the same reasoning as in \cite{FR}, this implies
\[
|u_2(t, x)| \leq C(t)\|u_2(0,\cdot)\|_{L^2(\Omega)} d^s(x),~~~\textrm{ for all } x \in \Omega.
\]
The constant $C$ depends only on $t, n, s, |\Omega|$, and the ellipticity constants \eqref{eq.L2}. This implies our result, since
\begin{align*}
\|u_2(0,\cdot)\|_{L^2(\Omega)} \leq |\Omega|\left( 1 + \|u_1\|_{L^\infty(\Omega)}\right) \leq C
\end{align*}
for $C$ depending only on $n, s, |\Omega|$ and the ellipticity constants \eqref{eq.L2}.
\end{proof}

\subsection{$C^s$ regularity up to the boundary} We begin this subsection by introducing a definition that will be useful through this and the next section.
\begin{defi}
\label{defi.1}
We say that $\Gamma$ is a $C^{1,1}$ surface with radius $\rho_0$ splitting $B_1$ into $\Omega^+$ and $\Omega^-$ if the following happens:
\begin{itemize}
\item The two disjoint domains $\Omega^+$ and $\Omega^-$ partition $B_1$, i.e., $\overline{B_1} = \overline{\Omega^+} \cup \overline{\Omega^-}$.
\item The boundary $\Gamma := \de \Omega^+\setminus \de B_1 = \de \Omega^- \setminus \de B_1$ is a $C^{1,1}$ surface with $0\in \Gamma$.
\item All points on $\Gamma \cap \overline{B_{3/4}}$ can be touched by two balls of radii $\rho_0$, one contained in $\Omega^+$ and the other contained in $\Omega^-$.
\end{itemize}
\end{defi}

Under the previous definition, we will denote
\[
d(x) = \textrm{dist}(x, \Omega^-).
\]

We will prove the following version of Proposition \ref{prop.mainboundary}.

\begin{prop}
\label{prop.bound}
Let $s\in (0,1)$ and let $L$ be an operator of the form \eqref{eq.L1}-\eqref{eq.L2}. Let $\Gamma$ be a $C^{1,1}$ surface with radius $\rho_0$ splitting $B_1$ into $\Omega^+$ and $\Omega^-$. Let $u$ be any weak solution to
\begin{equation}
\label{eq.fracdom}
  \left\{ \begin{array}{rcll}
  \de_t u - L u&=&f& \textrm{in }(0,1)\times \Omega^+\\
  u&=&0& \textrm{in }(0,1)\times \Omega^- . \\
  \end{array}\right.
\end{equation}

Then,
\begin{equation}
\|u\|_{C^{\frac{1}{2},s}_{t,x}\left(\left[\frac{1}{2},1\right]\times\overline{B_{1/2}}\right)}\leq C\left(\|f\|_{L^\infty((0,1)\times \Omega^+)} + \|u\|_{L^\infty((0,1)\times\R^n)}\right),
\end{equation}
where $C$ depends only on $n, s, \rho_0$ and the ellipticity constants \eqref{eq.L2}.
\end{prop}

To prove the previous proposition we will follow the steps of \cite[Proposition 1.1]{RS2}. We begin with the following lemma.
\begin{lem}
\label{lem.bds0}
Let $s\in(0,1)$ and let $L$ be an operator of the form \eqref{eq.L1}-\eqref{eq.L2}. Let $u$ be any weak solution to \eqref{eq.frac} with $f\in L^\infty((-1,0)\times B_1)$, and
\[
K_0 = \sup_{t\in(-1,0)}\sup_{R\geq 1} R^{\delta-2s}\|u(t,\cdot)\|_{L^\infty(B_R)}
\]
for some $\delta > 0$. Then, for any $\epsilon > 0$,
\[
\|u\|_{C^{1-\frac{\epsilon}{2s}}_t \left(\left[ -\frac{1}{2},0 \right]\times \overline{B_{1/2}}\right)} +  \|u\|_{C^{2s-\epsilon}_x\left(\left[ -\frac{1}{2},0 \right]\times \overline{B_{1/2}}\right)} \leq C(K_0+ \|f\|_{L^\infty((-1,0)\times B_1)}),
\]
where the constant $C$ depends only on $n, s, \epsilon, \delta$ and the ellipticity constants \eqref{eq.L2}.
\end{lem}
\begin{proof}
Apply Theorem \ref{thm.mainb} to $\tilde{u} = u\chi_{B_2}$. Then, by an argument similar to the one done in the proof of Theorem \ref{thm.mainb}, it is enough to check
\[
\|L (u (1-\chi_{B_2}))\|_{L^\infty((-1,0)\times B_1)} \leq C K_0.
\]
This is immediate from the growth imposed by the definition of $K_0$, i.e.,
\[
|u(t, x)|\leq K_0 \left(1+|x|^{2s-\delta}\right).
\]
Thus, the lemma follows.
\end{proof}

We next show that the solutions $u$ satisfy $|u|\leq Cd^s$.

\begin{lem}
\label{lem.bds}
Let $s\in (0,1)$ and let $L$ be an operator of the form \eqref{eq.L1}-\eqref{eq.L2}. Let $\Gamma$ be a $C^{1,1}$ surface with radius $\rho_0$ splitting $B_1$ into $\Omega^+$ and $\Omega^-$, and let $f\in L^\infty((0,1)\times \Omega^+)$. Let $u$ be any weak solution to \eqref{eq.fracdom}. Then
\[
|u(t, x)| \leq C(t_0) \left(\|f\|_{L^\infty((0,1)\times \Omega^+)} + \|u\|_{L^\infty((0,1)\times\R^n)}\right)d^s(x),
\]
for all $x\in B_{1/4}$, $t \geq t_0 > 0$, and where $C$ depends only on $t_0,n, s,\rho_0$ and the ellipticity constants \eqref{eq.L1}-\eqref{eq.L2}.
\end{lem}
\begin{proof}
Pick any point $z \in \Gamma\cap B_{1/2}$, and consider the ball $B^{(z)}$ tangent at $z$ and inside $\Omega^-$ with radius $\min\{\rho_0,\frac{1}{8}\}$. Then construct the supersolution from Lemma~\ref{lem.super} in the domain $B_2 \setminus \overline{B^{(z)}}$. This yields the desired result for points near $z$ with a constant $C$ that does not depend on the $z$ chosen. Repeating the argument for any point in $\Gamma\cap B_{1/2}$ we are done: indeed, for any $x\in B_{1/4}$ we apply this to $z_x\in\Gamma\cap B_{1/2}$ such that $d(x) = \textrm{dist}(z_x, x)$ and the result follows.
\end{proof}

As a consequence of the previous bound we find the following

\begin{lem}
\label{lem.bds2}Let $s\in (0,1)$ and let $L$ be an operator of the form \eqref{eq.L1}-\eqref{eq.L2}. Let $\Gamma$ be a $C^{1,1}$ surface with radius $\rho_0$ splitting $B_1$ into $\Omega^+$ and $\Omega^-$, $f\in L^\infty((0,1)\times \Omega^+)$ and $u$ be any weak solution to \eqref{eq.fracdom}. Then, for all $x_0\in \Omega^+\cap B_{1/4}$, and all $R \leq \frac{d(x_0)}{2}$,
\begin{equation}
[u]_{C^{\frac{1}{2}, s}_{t, x}\left(\left(t_1-\frac{1}{2}R^{2s}, t_1\right)\times \overline{B_R(x_0)}\right)} \leq C  \left(\|f\|_{L^\infty((0,1)\times\Omega^+)} + \|u\|_{L^\infty((0,1)\times\R^n)}\right),
\end{equation}
where $t_1$ is such that $\frac{1}{4}\leq t_1 -  R^{2s} < t_1 \leq 1$ (making $R$ smaller if necessary). The constant $C$ depends only on $n, s, \rho_0$ and the ellipticity constants \eqref{eq.L1}-\eqref{eq.L2}.
\end{lem}
\begin{proof}
Notice that $B_R(x_0) \subset B_{2R}(x_0)\subset \Omega^+$. Let $\tilde{u}(t, y) = u(R^{2s}t+t_1-R^{2s}, x_0+Ry)$, so that
\begin{equation}
\label{eq.bds2_1}
\de_t \tilde{u} - L\tilde{u} = R^{2s} f(R^{2s}t+t_1-R^{2s}, x_0+Ry)
\end{equation}
\[
 \textrm{ for } x\in B_1, t \geq \frac{R^{2s}-t_1}{R^{2s}}.
\]
We define
\[
C_0 := \|f\|_{L^\infty((0,1)\times \Omega^+)} + \|u\|_{L^\infty((0,1)\times\R^n)},
\]
and
\[
T_R := \frac{1}{R^{2s}}\left(\frac{1}{4}+R^{2s}-t_1\right).
\]
From Lemma \ref{lem.bds} with $t_0 = \frac{1}{4}$ we get
\begin{align}
\label{eq.bds2_2}
\|\tilde{u}\|_{L^\infty\left(\left(T_R,1\right)\times B_{1/4}\right)} \leq CC_0R^s.
\end{align}

Now note that, by Lemma \ref{lem.bds}, for all $y\in \R^n$,
\[|\tilde{u}(t, y)| \leq  CC_0 d^s(x_0+Ry) \leq C C_0R^s (1+|y|^s) \textrm{ for } t \in (T_R, 1).\]
Thus, we obtain
\begin{equation}
\label{eq.bds2_3}
\sup_{t\in (T_R, 1)} \sup_{r \geq 1} r^{-3s/2}\|\tilde{u}(t,\cdot)\|_{L^\infty(B_r)} \leq  C C_0R^s.
\end{equation}

Using Lemma \ref{lem.bds0} with $\epsilon = s$ and expressions \eqref{eq.bds2_1}-\eqref{eq.bds2_2}-\eqref{eq.bds2_3} we obtain
\[
\|\tilde{u}\|_{C^{\frac{1}{2}}_t \left(\left( \frac{1}{2},1 \right)\times \overline{B_{1/4}}\right)} +  \|\tilde{u}\|_{C^{s}_x\left(\left(\frac{1}{2},1 \right)\times \overline{B_{1/4}}\right)} \leq C(t_0) C_0R^s,
\]
where we have used that, under these hypotheses, $T_R \leq 0$.

Finally, use that
\[
R^{-s}[\tilde{u}]_{C^{\frac{1}{2},s}_{t, x}\left(\left(\frac{1}{2},1\right)\times \overline{B_{1/4}}\right)} = [u]_{C^{\frac{1}{2},s}_{t, x}\left(\left(t_1-\frac{1}{2}R^{2s},t_1\right)\times \overline{B_{R/4}(x_0)}\right)},
\]
to get
\[
[u]_{C^{\frac{1}{2},s}_{t, x}\left(\left(t_1-\frac{1}{2}R^{2s},t_1\right)\times \overline{B_{R/4}(x_0)}\right)} \leq  C C_0.
\]
By a standard covering argument, we find the desired result in $\overline{B_R(x_0)}$.
\end{proof}

We now prove Proposition \ref{prop.bound}.

\begin{proof}[Proof of Proposition \ref{prop.bound}]
Since
\[
[u]_{C^{\frac{1}{2},s}_{t,x}\left(\left[\frac{1}{2},1\right]\times(\overline{\Omega^+}\cap B_{1/2})\right)} \leq [u]_{C^{s}_{x}\left(\left[\frac{1}{2},1\right]\times(\overline{\Omega^+}\cap B_{1/2})\right)} + [u]_{C^{\frac{1}{2}}_{t}\left(\left[\frac{1}{2},1\right]\times(\overline{\Omega^+}\cap B_{1/2})\right)}
\]
we can treat these two terms separately.

For the first term we need to show
\begin{equation}
\label{eq.bound1_2}
\frac{|u(t, x)-u(t,x')|}{|x-x'|^{s}} \leq  C C_0,
\end{equation}
for any $t\in\left(\frac{1}{2},1\right)$, $x,x'\in \Omega^+\cap B_{1/2}$ and constant $C$ independent of $t_1$, where again we define
\[
C_0 =\|f\|_{L^\infty((0,1)\times\Omega^+)} + \|u\|_{L^\infty((0,1)\times\R^n)}.
\]
Let $r_d = \min\{d(x), d(x')\}$ and $R = |x-y|$. We now separate two cases according to the values of $r_d$ and $R$:

If $2R \geq r_d$, from Lemma~\ref{lem.bds} with $t_0 = 1/2$ we have
\[
|u(t, x)-u(t, x')| \leq CC_0\left(r_d^s+(R+r_d)^s\right) \leq CC_0 R^s
\]
so that \eqref{eq.bound1_2} is fulfilled.

On the other hand, if $2R < r_d$ and $x\in B_{1/4}$, then $B_{2R}(x)\subset\Omega$ and therefore, from Lemma~\ref{lem.bds2} we would get $[u]_{C^s_x(B_R(x)} \leq CC_0$. This can be extended for $x\in B_{1/2}$ using a covering argument. Thus, \eqref{eq.bound1_2} is proved.

For the second term in the seminorm we want to show
\begin{equation}
\label{eq.bound1_3}
\frac{|u(t, x_0)-u(t',x_0)|}{|t-t'|^{\frac{1}{2}}} \leq  CC_0,
\end{equation}
for any $x_0 \in \Omega^+\cap B_{1/4}$, $t, t'\in \left(\frac{1}{2},1\right)$. Again, this can be extended to $\Omega^+\cap B_{1/2}$ by a covering argument. Notice that we can suppose that $|t-t'|$ is small as long as it is independent of $x_0$. Let $\bar{x}\in \Omega^+$ to be chosen later, and observe that
\[
|u(t, x_0)-u(t',x_0)| \leq 2\sup_{t_*\in\left(\frac{1}{2},1\right)}|u(t_*, x_0)-u(t_*,\bar{x})| +|u(t, \bar{x})-u(t',\bar{x})|.
\]
By \eqref{eq.bound1_2} we have
\[
\sup_{t_*\in\left(\frac{1}{2},1\right)}|u(t_*, x_0)-u(t_*,\bar{x})| \leq CC_0 |x_0-\bar{x}|^s.
\]
Moreover, choosing $\bar x$ such that $|t-t'| \leq \frac{d(\bar{x})^{2s}}{2}$, by Lemma~\ref{lem.bds2} we have
\[
|u(t, \bar{x})-u(t',\bar{x})| \leq CC_0 |t-t'|^{\frac{1}{2}}.
\]

Therefore, choosing $\bar{x}$ such that
\[
\epsilon_0 |x_0-\bar{x}| \leq |t-t'|^{\frac{1}{2s}} \leq 2^{-\frac{1}{2s}} d(\bar{x})
\]
\eqref{eq.bound1_3} follows. Notice that such $\bar x$ and $\epsilon_0 > 0$ independent of $x_0$, $\bar x$, $t$ and $t'$ always exist if $|t-t'|$ is small enough, depending on $\epsilon_0$ and the $C^{1,1}$ norm of the domain.
\end{proof}

Proposition~\ref{prop.bound} directly yields Proposition~\ref{prop.mainboundary}.

\begin{proof}[Proof of Proposition \ref{prop.mainboundary}]
The result follows combining Proposition \ref{prop.bound} with the interior estimates of Theorem~\ref{thm.mainb}.
\end{proof}

We next present an immediate consequence of Proposition~\ref{prop.bound} analogous to Lemma \ref{lem.bds0} but for the case with boundary that will be useful later (and that is why we consider the temporal domain to be $(-1,0)$ now).

\begin{cor}
\label{cor.bds0}
Let $s\in (0,1)$ and let $L$ be an operator of the form \eqref{eq.L1}-\eqref{eq.L2}. Let $\Gamma$ be a $C^{1,1}$ surface with radius $\rho_0$ splitting $B_1$ into $\Omega^+$ and $\Omega^-$. Suppose that $f\in L^\infty((-1,0)\times \Omega^+)$ and $u$ is any weak solution to
\begin{equation}
  \left\{ \begin{array}{rcll}
  \de_t u - L u&=&f& \textrm{in }(-1,0)\times \Omega^+\\
  u&=&0& \textrm{in }(-1,0)\times \Omega^- . \\
  \end{array}\right.
\end{equation}

Define
\[
K_0 = \sup_{t\in(-1,0)} \sup_{R\geq 1} R^{\delta-2s}\|u(t,\cdot)\|_{L^\infty(B_R)},
\]
for some $\delta > 0$. Then,
\[
\|u\|_{C^{\frac{1}{2}}_t \left(\left[ -\frac{1}{2},0 \right]\times \overline{B_{1/2}}\right)} +  \|u\|_{C^{s}_x\left(\left[ -\frac{1}{2},0 \right]\times \overline{B_{1/2}}\right)} \leq C(K_0+ \|f\|_{L^\infty((-1,0)\times \Omega^+)}).
\]
where $C$ depends only on $n, s, \rho_0, \delta$ and the ellipticity constants \eqref{eq.L2}.
\end{cor}

\begin{proof}
The proof is the same as the proof of Lemma \ref{lem.bds0}, using Proposition \ref{prop.bound}. Indeed, define $\tilde{u} = u\chi_{B_2}$ and notice that
\begin{align*}
\|L\tilde u\|_{L^\infty((-1,0)\times B_1)} & \leq \|Lu\|_{L^\infty((-1,0)\times B_1)}+\|L\left(u(1-\chi_{B_2})\right)\|_{L^\infty((-1,0)\times B_1)} \\
& \leq \|f\|_{L^\infty((-1,0)\times \Omega^+)} + CK_0,
\end{align*}
which follows from the growth imposed by the definition of $K_0$.
\end{proof}

\subsection{Liouville-type theorem in the half space} We now prove a Liouville-type theorem in the half space for nonlocal parabolic equations.

\begin{thm}
\label{thm.liouv2}
Let $s\in(0,1)$, and let $L$ be any operator of the form \eqref{eq.L1}-\eqref{eq.L2}. Let $u$ be any weak solution of
\begin{equation}\label{eq.halfspace}
  \left\{ \begin{array}{rcll}
  \de_t u - L u&=&0& \textrm{in }(-\infty,0)\times \R^n_+ \\
  u&=&0& \textrm{in } (-\infty,0)\times \R^n_- ,\\
  \end{array}\right.
\end{equation}
such that
\[
\|u(t, \cdot)\|_{L^\infty(B_R)} \leq C\left(R^\gamma+1\right) \textrm{ for } R \geq |t|^{\frac{1}{2s}},
\]
for some $\gamma < 2s$. Then,
\[
u(t, x) = K(x_n)^s_+
\]
for some constant $K\in \R$.
\end{thm}

\begin{proof}
We proceed as in the proof of \cite[Theorem 4.1]{RS}.

Given $\rho > 0$ define $v_\rho (t, x) = \rho^{-\gamma} u(\rho^{2s}t, \rho x)$. Then,
\begin{equation}\label{eq.halfspace2}
  \left\{ \begin{array}{rcll}
  \de_t v_\rho - L v_\rho&=&0& \textrm{in }(-\infty,0)\times \R^n_+ \\
  v_\rho&=&0& \textrm{in } (-\infty,0)\times \R^n_- ,\\
  \end{array}\right.
\end{equation}
and for $R \geq |t|^{\frac{1}{2s}}$,
\begin{align}
\label{eq.boundliouv}
\|v_\rho(t, \cdot)\|_{L^\infty(B_R)} = \rho^{-\gamma} \|u(\rho^{2s}t, \cdot)\|_{L^\infty(B_{\rho R})} \leq \rho^{-\gamma} C(1+(\rho R)^\gamma)\leq  C\left(1+R^\gamma\right) \\ \nonumber \textrm{ for } \rho \geq 1.
\end{align}

Hence, denoting $\overline{v}_\rho = v_\rho \chi_{B_2}(x)$, we have that $\overline{v}_\rho\in L^\infty((-1,0)\times \R^n)$ satisfies
\begin{equation}
  \left\{ \begin{array}{rcll}
  \de_t \overline{v}_\rho - L \overline{v}_\rho&=&g_\rho& \textrm{in }(-1,0)\times B_1^+ \\
  \overline{v}_\rho&=&0& \textrm{in } (-1,0)\times B_1^-,\\
  \end{array}\right.
\end{equation}
for some $g_\rho \in L^\infty((-1,0)\times B_1^+)$, with $\|g_\rho\|_{L^\infty((-1,0)\times B_1^+)} \leq C_0$, for some constant $C_0$ independent of $\rho\geq 1$. The constant $C_0$ depends only on the constant $C$ in \eqref{eq.boundliouv}. By Proposition \ref{prop.bound} we find
\[
\|v_{\rho}\|_{C^{\frac{1}{2},s}_{t,x}\left(\left(-2^{-2s}, 0\right)\times B_{1/2}\right)} = \|\overline{v}_{\rho}\|_{C^{\frac{1}{2},s}_{t,x}\left(\left(-2^{-2s}, 0\right)\times B_{1/2}\right)} \leq C C_0.
\]
Therefore, for $\rho \geq 1$
\begin{align*}
[u]_{C^{\frac{1}{2},s}_{t,x}\left(\left(-2^{-2s}\rho^{2s}, 0\right)\times B_{\rho/2}\right)} & = \rho^{-s}[u(\rho^{2s}t, \rho x)]_{C^{\frac{1}{2},s}_{t,x}\left(\left(-2^{-2s}, 0\right)\times B_{1/2}\right)}\\
& = \rho^{s-\gamma}[v_{\rho}]_{C^{\frac{1}{2},s}_{t,x}\left(\left(-2^{-2s}, 0\right)\times B_{1/2}\right)} \leq C C_0 \rho^{\gamma-s}.
\end{align*}

Now, given $h \in B_1$ with $h_n = 0$, and $\tau \in (-1,0)$ consider
\[
w_h (t, x) = \frac{u(t+\tau, x+h)-u(t, x)}{\tau^{\frac{1}{2}}+|h|^s},
\]
so that, by the previous result, whenever $R \geq |t|^{\frac{1}{2s}}$ we have that
\[
\|w_h(t, \cdot)\|_{L^\infty(B_R)} \leq C (R^{\gamma-s}+1).
\]
By linearity $\de_t w_h - L w_h = 0$ in $(-\infty,0)\times \R^n_+$ and $w = 0$ in $(-\infty,0)\times \R^n_-$. We can then apply the previous reasoning with $u$ replaced by $w_h$ to finally reach that
\begin{align*}
[w_h]_{C^{1/2,s}_{t,x}\left(\left(-2^{-2s}\rho^{2s}, 0\right)\times B_{\rho/2}\right)} \leq C C_0 \rho^{\gamma-2s}, \textrm{ for } \rho \geq 1.
\end{align*}

Since $2s > \gamma$, making $\rho \to \infty$ we find that $w_h$ must be constant. But since $w_h = 0$ in $\R^n_-$ then
\[
w_h \equiv 0 \textrm{ in } (-\infty,0)\times \R^n.
\]

This implies that for all $h \in B_1$ with $h_n = 0$ and for all $\tau \in (-1,0)$, then $u(t+\tau, x+h) = u(t, x) $. Thus, $u$ is constant in time, and by \cite[Theorem 4.1]{RS} we get $u(t, x) = K(x_n)_x^s$ as desired. Alternatively, we could end the proof by noticing that
\[
u(t, x) = \overline{u}(x_n)
\]
for some 1D function $\overline{u}$, and proceeding as in the final part of the proof of \cite[Theorem 4.1]{RS}.
\end{proof}

\section{Regularity up to the boundary for $u/d^s$}
\label{sec.5}
In this section we will prove Theorem \ref{thm.mainboundary}. We begin by introducing a definition that will be recurrent throughout the section.

Let $s\in (0,1)$ and let $L$ be an operator of the form \eqref{eq.L1}-\eqref{eq.L2}. Let $\Gamma$ be a $C^{1,1}$ surface with radius $\rho_0$. Under the notation in Definition \ref{defi.1}, and when not specified otherwise, in the whole section we will define $\bar{u} = \bar{u}(x)$ as any solution to
\begin{equation}
\label{eq.ubareq}
  \left\{ \begin{array}{rcll}
  L \bar{u}&=&1& \textrm{in }\Omega^+\\
  \bar{u}&=&0& \textrm{in } \Omega^- , \\
  0~\leq~\bar{u}&\leq&c_2& \textrm{in } \R^n\setminus B_1 , \\
  \end{array}\right.
\end{equation}
where $c_2$ is a constant depending only on $n$, $s$ and the ellipticity constants.

Notice that, under these circumstances, we have
\begin{equation}
\label{eq.ubareq2}
0 < c_0 d^s \leq \bar{u} \leq c_1 d^s \textrm{ in } B_{1/2}
\end{equation}
where $c_0$ and $c_1$ are constants depending only on $n$, $s$, $\rho_0$ and the ellipticity constants \eqref{eq.L2}. The first inequality in \eqref{eq.ubareq2} appears, for example, in \cite[Lemma 7.4]{R}, while the second one is a consequence of the $C^s$ regularity up to the boundary for this elliptic problem (see for example \cite[Proposition 4.6]{RS} or the previous section).

The result we will need before proving Theorem~\ref{thm.mainboundary} is the following.

\begin{prop}
\label{prop.mainbound}
Let $s\in(0,1)$ and $\gamma\in (s, 2s)$. Let $L$ be an operator of the form \eqref{eq.L1}-\eqref{eq.L2}, let $\Gamma$ be a $C^{1,1}$ surface with radius $\rho_0$ splitting $B_1$ into $\Omega^+$ and $\Omega^-$, and let $\bar{u}$ be a function satisfying \eqref{eq.ubareq}.

Let $u$ be any weak solution to
\begin{equation}
  \left\{ \begin{array}{rcll}
  \de_t u - L u&=&f& \textrm{in }(-1,0)\times \Omega^+\\
  u&=&0& \textrm{in }(-1,0)\times \Omega^- . \\
  \end{array}\right.
\end{equation}
and define
\[
C_0 = \|u\|_{L^\infty((-1,0)\times\R^n)}+\|f\|_{L^\infty((-1,0)\times \Omega^+)}.
\]
Then, there is a constant $Q\in \R$ with $|Q|\leq CC_0$ for which
\[
|u(t, x)-Q\bar{u}(x)| \leq C C_0 \left( |x|^\gamma + |t|^{\frac{\gamma}{2s}}\right) \textrm{ in } (-1,0)\times B_1.
\]
The constant $C$ depends only on $n$, $\rho_0$, $s$, $\gamma$ and the ellipticity constants \eqref{eq.L2}.
\end{prop}

In order to prove this proposition we will need the following lemma.

\begin{lem}
\label{lem.boundreg}
Let $s\in(0,1)$, $\gamma > s$, and $u\in C((-1,0)\times B_1)$. Let $L$ be an operator of the form \eqref{eq.L1}-\eqref{eq.L2} and let $\Gamma$ be a $C^{1,1}$ surface splitting $B_1$ into $\Omega^+$ and $\Omega^-$, with radius $\rho_0$. Let $\bar{u}$ be a solution to \eqref{eq.ubareq}.
Define
\[
\phi_r(x) := Q_*(r) \bar{u}(x),
\]
\[
Q_*(r) := \arg \min_{Q\in R} \int_{-r^{2s}}^0 \int_{B_r} \left(u(t, x) - Q\bar{u} \right)^2dxdt = \frac{\int_{-r^{2s}}^0 \int_{B_r} u(t, x) \bar{u} dx dt}{r^{2s}\int_{B_r}\bar{u}^2dx}.
\]

Assume that for all $r\in (0,1)$ we have that
\[
\| u-\phi_r \|_{L^\infty((-r^{2s},0)\times B_r)} \leq C_* r^\gamma.
\]
Then, there is $Q\in \R$ satisfying $|Q| \leq C(C_* + \|u\|_{L^\infty((-1,0)\times B_1)})$ such that
\[
\|u-Q\bar{u}\|_{L^\infty((-r^{2s},0)\times B_r)} \leq C C_* r^\gamma
\]
for some constant $C$ depending only on $\gamma$, $n$, $s$, $\rho_0$ and the ellipticity constants \eqref{eq.L2}.
\end{lem}
\begin{proof}
Notice that by \eqref{eq.ubareq2} we have
\begin{equation}
\label{eq.ubar2}
\|\bar{u}\|_{L^\infty(B_r)} \geq c r^s \textrm{ for } r\in(0,1),
\end{equation}
where $c$ is a constant depending only on $n, s, \rho_0$ and the ellipticity constants \eqref{eq.L2}.

On the other hand, notice that
\[
|Q_*(1)| \leq \frac{\|u\|_{L^\infty((-1,0)\times B_1)}\|\bar{u}\|_{L^\infty(B_1)}}{\|\bar{u}\|^2_{L^2(B_1)}}.
\]

By \eqref{eq.ubareq2} again $\|\bar{u}\|_{L^\infty(B_1)} \leq C$, and since there is some ball of radius $\rho_0$ touching the origin inside $\Omega^+$ where $\bar{u} \geq c_1 d^s$, we have that $\|\bar{u}\|^2_{L^2(B_1)} \geq C'$, for some constants $C, C'$ depending only on $n, s, \rho_0$ and the ellipticity constants \eqref{eq.L2}. Thus
\begin{equation}
\label{eq.Qstar}
|Q_*(1)| \leq C\|u\|_{L^\infty((-1,0)\times B_1)},
\end{equation}
for some $C$ depending only on $n, s, \rho_0$ and the ellipticity constants \eqref{eq.L2}.

Using \eqref{eq.ubar2} and \eqref{eq.Qstar}, the proof is exactly the same as the proof of \cite[Lemma 5.3]{RS}.
\end{proof}

We now prove Proposition~\ref{prop.mainbound}. The proof is by contradiction, and uses some ideas from \cite[Proposition 5.2]{RS}.

\begin{proof}[Proof of Proposition \ref{prop.mainbound}]
 Assume that there are sequences $\Gamma_k$, $\Omega_k^+$, $\Omega_k^-$, $f_k$, $u_k$ and $L_k$ satisfying the hypotheses of the proposition. That is,
\begin{itemize}
\item $\Gamma_k$ is a $C^{1,1}$ surface with radius $\rho_0$ splitting $B_1$ into $\Omega_k^+$ and $\Omega_k^-$, and we assume without loss of generality that the normal vector to $\Gamma$ at the origin is~$e_n$.
\item $L_k$ is of the form \eqref{eq.L1}-\eqref{eq.L2}.
\item $\|u_k\|_{L^\infty((-1,0)\times\R^n)}+\|f_k\|_{L^\infty((-1,0)\times \Omega_k^+)} \leq 1$.
\item{ $u_k$ is a solution of
\begin{equation}
  \left\{ \begin{array}{rcll}
  \de_t u_k - L_k u_k&=&f_k& \textrm{in }(-1,0)\times \Omega_k^+\\
  u_k&=&0& \textrm{in }(-1,0)\times \Omega_k^- . \\
  \end{array}\right.
\end{equation}
}
\end{itemize}

In order to reach a contradiction, suppose that the conclusion of the proposition does not hold. That is, for all $C > 0$, there are $k$ and $\bar{u}_k$ for which no constant $Q\in \R$ satisfies
\begin{equation}
\label{eq.negmainbound}
|u_k(t, x)-Q\bar{u}_k(x)| \leq C \left( |x|^\gamma + |t|^{\frac{\gamma}{2s}}\right) \textrm{ in } (-1,0)\times B_1,
\end{equation}
where $\bar{u}_k$ solves
\begin{equation}
  \left\{ \begin{array}{rcll}
  L_k \bar{u}_k&=&1& \textrm{in }\Omega_k^+\\
  \bar{u}_k&=&0& \textrm{in } \Omega_k^-, \\
    0~ \leq~~  \bar{u}_k&\leq&c_2& \textrm{in } \R^n\setminus B_1. \\
  \end{array}\right.
\end{equation}

We will divide the proof by contradiction into four steps.
\\[0.5cm]
%*****************************STEP 1************************
{\bf Step 1: The blow-up sequence.}
Notice that, by Lemma \ref{lem.boundreg}, and the negation of \eqref{eq.negmainbound}, we have
\begin{equation}
\sup_k~\sup_{r > 0}~~  r^{-\gamma}\|u_k-\phi_{k,r}\|_{L^\infty((-r^{2s},0)\times B_r)} =\infty,
\end{equation}
where
\begin{equation}
\phi_{k,r} (x) = Q_{k}(r) \bar{u}_k,
\end{equation}
\[
Q_{k}(r) := \arg \min_{Q\in R} \int_{-r^{2s}}^0 \int_{B_r} \left(u_k(t, x) - Q\bar{u}_k \right)^2dxdt = \frac{\int_{-r^{2s}}^0 \int_{B_r} u_k(t, x) \bar{u}_k dx dt}{r^{2s}\int_{B_r}\bar{u}_k^2dx}.
\]
We define the following monotone decreasing function in $r$,
\[
\theta(r) := \sup_k~\sup_{r' > r}~~( r')^{-\gamma}\|u_k-\phi_{k,r'}\|_{L^\infty((-(r')^{2s},0)\times B_{r'})}.
\]
Notice that, $\theta(r) < \infty$ for $r > 0$ and $\theta(r) \uparrow \infty$
 as $r\downarrow 0$. Pick a sequence $r_m$, $k_m$ such that $r_m \geq \frac{1}{m}$ and
\begin{equation}
\label{eq.boundbelow}
r_m^{-\gamma}\|u_{k_m}-\phi_{k_m,r_m}\|_{L^\infty((-r_m^{2s},0)\times B_{r_m})} \geq \frac{\theta(1/m)}{2} \geq  \frac{\theta(r_m) }{2}.
\end{equation}
Notice that $r_m \downarrow 0$ as $m\to \infty$. To simplify notation we will denote $\phi_m = \phi_{k_m,r_m}$.

We now consider a blow-up sequence
\[
v_m(t, x) = \frac{u_{k_m}(r_m^{2s}t, r_m x)-\phi_m(r_m x)}{r_m^\gamma \theta(r_m)}.
\]
In the next step we analyse some properties of this blow-up sequence.
\\[0.5cm]
%***********************STEP 2***************************
{\bf Step 2: Properties of the blow-up sequence.} By the optimality condition for least squares we have that, for $m \geq 1$,
\begin{equation}
\label{eq.contr}
\int_{-1}^0\int_{B_1} v_m(t, x) \bar{u}_{k_m}(r_m x) dx dt = 0.
\end{equation}
Moreover,
\begin{equation}
\label{eq.growcont0}
\|v_m\|_{L^\infty((-1,0)\times B_1)} \geq 1/2,
\end{equation}
which is an immediate consequence of the expression \eqref{eq.boundbelow}.

In addition, for all $k$ we have that
\[
|Q_k(2r)-Q_k(r)|  \leq C r^{s-\gamma}\theta(r),
\]
for some $C$ depending only on $n$ and $s$.
Indeed,
\begin{align*}
|Q_k(2r)-Q_k(r)|& = \frac{\|\phi_{k,2r}-\phi_{k,r}\|_{L^\infty(B_r)}}{\|\bar{u}_k\|_{L^\infty(B_r)}}\\
& \leq Cr^{-s}\left( \|\phi_{k,2r}-u_k\|_{L^\infty((-(2r)^{2s},0)\times B_{2r})}+\|\phi_{k,r}-u_k\|_{L^\infty((-r^{2s},0)\times B_{r})}\right)\\
& \leq Cr^{-s} \left((2r)^\gamma\theta(2r) + r^\gamma \theta(r)\right) \leq Cr^{\gamma-s} \theta(r),
\end{align*}
where we have used that by \eqref{eq.ubareq2} for $r \leq \frac{1}{2}$,
\[
\|\bar{u}_k\|_{L^\infty(B_r)} \geq c_0 r^{s}
\]
for $c_0$ depending only on $n, s, \rho_0$ and the ellipticity constants \eqref{eq.L2}.

Thus, for $R = 2^N$ we have
\begin{align*}
\frac{r^{s-\gamma}|Q_k(rR)-Q_k(r)|}{\theta(r)} & \leq \sum_{j = 0}^{N-1} 2^{j(\gamma-s)} \frac{(2^jr)^{s-\gamma}|Q_k(2^{j+1}r)-Q_k(2^jr)|}{\theta(r)} \\
& \leq C \sum_{j = 0}^{N-1} 2^{j(\gamma-s)} \frac{\theta(2^j r)}{\theta(r)} \leq C 2^{N(\gamma-s)} = CR^{\gamma-s},
\end{align*}
for some $C$ depending only on $n$ and $s$.

Using this, we bound the growth of $v_m$,
\begin{align*}
\|v_m\|_{L^\infty((-R^{2s},0)\times B_R)}& = \frac{1}{r_m^\gamma \theta(r_m)}\|u_{k_m} - Q_{k_m}(r_m) \bar{u}_{k_m} \|_{L^\infty((-R^{2s}r_m^{2s}, 0)\times B_{Rr_m})}\\
& \leq \frac{R^\gamma}{(Rr_m)^\gamma \theta(r_m)}\|u_{k_m} - Q_{k_m}(Rr_m) \bar{u}_{k_m} \|_{L^\infty((-R^{2s}r_m^{2s}, 0)\times B_{Rr_m})} + \\
&~~~~~~~~~~~~~~~~ +\frac{1}{r_m^\gamma \theta(r_m)}|Q_{k_m}(Rr_m)- Q_{k_m}(r_m)  | (Rr_m )^s\\
& \leq \frac{R^\gamma \theta(Rr_m)}{\theta(r_m)}+ CR^\gamma.
\end{align*}
We have used here that by \eqref{eq.ubareq2}
\[
\|\bar{u}_k\|_{L^\infty(B_r)} \leq c_2r^{s}
\]
for some constant $c_2$ depending only on $n, s, \rho_0$ and the ellipticity constants \eqref{eq.L2}.

Therefore, we have the following growth control on $v_m$,
\begin{equation}
\label{eq.growcont}
\|v_m\|_{L^\infty((-R^{2s},0)\times B_R)} \leq C R^\gamma~\textrm{ for }~R \geq 1.
\end{equation}

Finally, notice that $v_m$ satisfy
\begin{align}
\label{eq.blowupeq}
\de_t v_m(t, x) - L_{k_m} v_m(t, x) = \frac{r_m^{2s-\gamma}}{\theta(r_m)}  \left( f_{k_m}(r_m^{2s}t,r_m x) -Q_{k_m}(r_m)\right)\\
\nonumber  \textrm{ in } (-R^{2s},0)\times\Omega^+_{R,m},
\end{align}
for all $0\leq R\leq r_m^{-1}$ and
\[
\Omega^+_{R,m} := \{x\in B_{R} : r_m x \in \Omega^+_{k_m}\}.
\]
\\
%****************************STEP 3***************************
{\bf Step 3: Convergence properties.} We next show that there is a subsequence of $v_m$ converging to some function $v$.

Notice that the right hand side of \eqref{eq.blowupeq} is uniformly bounded with respect to $m$. Indeed,
\begin{align*}
|Q_{k_m}(r_m)|  \leq \frac{\int_{-r_m^{2s}}^0 \int_{B_{r_m}}\left| u_{k_m}(t, x) \right|\bar{u}_{k_m} dx dt}{r_m^{2s}\int_{B_{r_m}}\bar{u}_{k_m}^2dx}
 \leq  \frac{Cc_1 \int_{B_{r_m}}Cd^s(x)c_2d^s(x)dx}{c_0 \int_{B_{r_m}}d^{2s}(x)dx} \leq C
\end{align*}
for some constant $C$ depending only on $n, s, \rho_0$ and the ellipticity constants \eqref{eq.L2}. Here we used \eqref{eq.ubareq2} and also that
\[
\sup_{t\in (-r^{2s},0)}|u_k(t,x)| \leq Cd^s(x),
\]
for $r$ small enough, which  follows by the $C^s$ regularity of Proposition \ref{prop.bound}.

Hence, using also that $\gamma < 2s$ and $\theta(r_m)\to \infty$, we find
\begin{equation}
\label{eq.grow2}
\|\de_t v_m(t, x) - L_{k_m} v_m(t, x)\|_{L^\infty((-R^{2s},0)\times\Omega^+_{R,m})} \leq \frac{r_m^{2s-\gamma}}{\theta(r_m)} \to 0 ~~~ \textrm{ as } m\to \infty.
\end{equation}
Thanks to the control \eqref{eq.growcont} and the bound from \eqref{eq.grow2} we can apply Corollary \ref{cor.bds0} with $\delta = 2s-\gamma > 0$ on domains of the form $(-R^{2s},0)\times B_{R}$, to obtain that
\begin{equation}
\|v_m\|_{C^{\frac{1}{2},s}_{t,x}((-R^{2s}/2,0)\times B_{R/2})} \leq C(R),
\end{equation}
for some constant $C$ depending only on $R, n, s,\rho_0$ and the ellipticity constants \eqref{eq.L2}. It is important to highlight that the dependence is on $\rho_0$ independent of $r_m$, and this is because the domains of the form $\Omega^+_{R,m}$ are $C^{1,1}$ surfaces with radius $\rho_0/r_m > \rho_0$.

Therefore, by the Arzelà-Ascoli theorem there is some subsequence of $v_m$ converging to some function $v$ uniformly over compact sets, since $(-R^{2s},0)\times B_{R}$ can be made arbitrarily large.

On the other hand, recall that from the compactness of probability measures on the sphere we can find a subsequence of $\{L_{k_m}\}$ converging weakly to an operator $\tilde{L}$ of the form \eqref{eq.L1}-\eqref{eq.L2}.

Now, consider any point $x \in \R^n_+$. The normal vector to $\Gamma_{k_m}$ at the origin is $e_n$ and there is a ball of radius $\rho_0/r_m$ contained in $\Omega^+_{r_m^{-1},m}$ tangent to $\Gamma_{k_m}$ at the origin. Therefore, for $m$ large enough we will have that $x\in \Omega^+_{r_m^{-1},m}$ eventually, and the same will happen for any neighbourhood of $x$ inside $\R^n_+$. Similarly, if $x\in \R^n_-$, for $m$ large enough we will have $v_m(t, x) = 0$ for any $t\in (-r_m^{2s},0)$. By Lemma \ref{lem.1}, we have that up to a subsequence $v_m$ converges locally uniformly to some $v$ satisfying
\begin{equation}\label{eq.grow5}
  \left\{ \begin{array}{rcll}
  \de_t v - \tilde{L} v&=&0& \textrm{in }(-\infty,0)\times \R^n_+\\
  v&=&0& \textrm{in }(-\infty,0)\times \R^n_- , \\
  \end{array}\right.
\end{equation}
for some operator $\tilde{L}$ of the form \eqref{eq.L1}-\eqref{eq.L2}. Furthermore, by uniform convergence and from \eqref{eq.growcont0}-\eqref{eq.growcont}, we have
\begin{equation}
\label{eq.grow3}
\|v\|_{L^\infty((-1,0)\times B_1)} \geq 1/2,
\end{equation}
and
\begin{equation}
\label{eq.grow4}
\|v\|_{L^\infty((-R^{2s},0)\times B_R)} \leq C R^\gamma \textrm{ for } R \geq 1.
\end{equation}

Finally, observe that $\frac{\bar{u}_{k_m}(r_mx)}{r_m^s}$ converges uniformly in $B_1$ up to a subsequence to $\bar{Q}(x_n)^s_+$ for some $\bar{Q} \in \R^+$. Indeed, by \cite[Proposition 5.2]{RS}, for each $m\in \N$ and $\gamma' \in (s, 2s)$ there is some $\bar{Q}_m$ such that
\begin{equation}
\label{eq.Qmbound}
|\bar{u}_{k_m}(x) - \bar{Q}_m (x_n)^s_+| \leq C|x|^{\gamma'}, \textrm{ for } x\in B_1,
\end{equation}
where $C$ depends only on $n, s, \rho_0, \gamma'$ and the ellipticity constants \eqref{eq.L2}. Moreover, thanks to \eqref{eq.ubareq2},
\begin{equation}
\label{eq.boundq}
0< c_0 \leq \bar{Q}_m \leq c_1 \textrm{ for all } m \in \N,
\end{equation}
where the constants $c_0$ and $c_1$ are the same as in \eqref{eq.ubareq2}. To check this, write for example
\[
c_0 d^s - \bar Q_m(x_n)^s_+ \leq |\bar u _{k_m}(x) - \bar Q_m(x_n)^s_+| \leq |x|^{\gamma'}.
\]
Now dividing the expression by $|(x_n)_+|^s$ and taking the limit for $x = he_n$ and $h\downarrow 0$ we would get $c_0 \leq \bar Q_m$. It similarly follows $\bar Q_m \leq c_1$.

Rescaling \eqref{eq.Qmbound},
\[
\sup_{x\in B_1} \left| \frac{\bar{u}_{k_m}(r_mx)}{r_m^s}-\bar{Q}_m(x_n)^s_+\right| \leq Cr_m^{\gamma'-s} \to 0 \textrm{ as } m \to\infty,
\]
and up to a subsequence we have that
\begin{equation}
\label{eq.unifconv}
\frac{\bar{u}_{k_m}(r_mx)}{r_m^s} \to \bar{Q}(x_n)^s_+ \textrm{ uniformly in } B_1
\end{equation}
for some $\bar{Q}$ fulfilling the same bounds as $\bar{Q}_m$, \eqref{eq.boundq}.
\\[0.5cm]
%******************************STEP 4***************************
{\bf Step 4: Contradiction.} By considering the expression \eqref{eq.contr} in the limit, and using $\frac{\bar{u}_{k_m}(r_mx)}{r_m^s} \to \bar{Q}(x_n)^s_+ \textrm{ uniformly in } B_1$, we have
\begin{equation}
\label{eq.contr2}
\int_{-1}^0 \int_{B_1} v(t, x) (x_n)^s_+ dx dt = 0.
\end{equation}

On the other hand, by \eqref{eq.grow5}-\eqref{eq.grow4} we can apply the Liouville-type theorem in the half space, Theorem \ref{thm.liouv2}, to $v$. Therefore, we have
\[
v(t, x) = k (x_n)^s_+, \textrm{ for some } k\in \R.
\]
By \eqref{eq.contr2}, $v \equiv 0$. However, this is not possible by \eqref{eq.grow3}, and we have reached a contradiction.
\end{proof}

Before proceeding to give the proof of Theorem~\ref{thm.mainboundary}, we state the following useful lemma. Thanks to this lemma, we can replace $\bar u$ by $d^s$ in the expression of Proposition~\ref{prop.mainbound}.

\begin{lem}[\cite{RS}]
\label{lem.rs}
Let $\Gamma$ be a $C^{1,1}$ surface with radius $\rho_0$ splitting $B_1$ into $\Omega^+$ and $\Omega^-$. Let $\bar{u}$ be a solution to \eqref{eq.ubareq}, and let $d(x) = {\rm dist} (x, \Omega^-)$. Let $x_0\in B_{1/2}$ such that
\[
{ \rm dist} (x_0,\Gamma) = {\rm dist} (x_0, z) =: 2r < \rho_0.
\]
Then, there exists some $\bar{Q} = \bar{Q}(z)$ such that $|\bar{Q}(z)| \leq C$,
\begin{equation}
\label{eq.rs1}
\|\bar{u}-\bar{Q}d^s\|_{L^\infty(B_r(x_0))} \leq Cr^{2s-\epsilon},
\end{equation}
and
\begin{equation}
\label{eq.rs2}
 [\bar{u}-\bar{Q}d^s]_{C^{s-\epsilon}_x (B_r(x_0))} \leq Cr^{s},
\end{equation}
for some constant $C$ depending only on $n, s, \epsilon, \rho_0$ and the ellipticity constants \eqref{eq.L2}. Moreover,
\begin{equation}
\label{eq.rs3}
[d^{-s}]_{C^{s-\epsilon}(B_r(x_0))} \leq C^*r^{-2s+\epsilon}
\end{equation}
for some constant $C^*$ depending only on $\rho_0$.
\end{lem}

\begin{proof}
The third expression \eqref{eq.rs3} is in \cite[Lemma 5.5]{RS}, while the first two expressions, \eqref{eq.rs1}-\eqref{eq.rs2}, appear in the proof \cite[Theorem 1.2]{RS}.
\end{proof}

Finally, we can proceed with the proof of Theorem~\ref{thm.mainboundary}. We will prove first the following proposition, which is essentially the same but assuming $0\in \de\Omega$.

\begin{prop}
\label{prop.mainbound2}
Let $s\in(0,1)$, and let $\Gamma$ be a $C^{1,1}$ surface with radius $\rho_0$ splitting $B_1$ into $\Omega^+$ and $\Omega^-$. Let $u$ be a weak solution to
\begin{equation}
  \left\{ \begin{array}{rcll}
  \de_t u - L u&=&f& \textrm{in }(-1,0)\times \Omega^+\\
  u&=&0& \textrm{in }(-1,0)\times \Omega^- . \\
  \end{array}\right.
\end{equation}
where $L$ is an operator of the form \eqref{eq.L1}-\eqref{eq.L2}. Let
\[
C_0 = \|u\|_{L^\infty((-1,0)\times\R^n)}+\|f\|_{L^\infty((-1,0)\times \Omega^+)}.
\]

Then, for any $\epsilon > 0$,
\begin{equation}
\|u\|_{C^{1-\epsilon}_t\left(\left(-\frac{1}{2},0\right)\times \overline{B_{1/2}}\right)} + \left\| u/d^s \right\|_{C^{\frac{1}{2}-\frac{\epsilon}{2s},s-\epsilon}_x\left(\left(-\frac{1}{2},0\right)\times(\overline{\Omega^+}\cap {B_{1/2}})\right)} \leq C C_0,
\end{equation}
where the constant $C$ depends only on $\epsilon$, $n$, $\rho_0$, $s$  and the ellipticity constants \eqref{eq.L2}.
\end{prop}

\begin{proof}
We may assume that
\[
\|u\|_{L^\infty((-1,0)\times \R^n)}+\|f\|_{L^\infty((-1,0)\times \Omega^+)} \leq 1.
\]
Also, inside $\Omega$ the result follows from the interior regularity, so we only need to show the estimates in $\Omega^+\cap\{x:d(x)< \rho_0\}$.

Pick a point $x_0 \in B_{1/4}\cap \Omega^+\cap\{x:d(x)< \rho_0\}$, and consider $z\in \Gamma$ minimising the distance to $x_0$, i.e.,
\begin{equation}
\label{eq.mainthm0}
2r := \textrm{dist}(x_0,\Gamma) = \textrm{dist}(x_0,z) < \rho_0.
\end{equation}

During the proof we will assume that $r$ is as small as we need (namely, $4r^{2s} < 1$), as long as it does not depend on $x_0$. Under these assumptions $B_r(x_0) \subset B_{2r}(x_0) \subset \Omega^+$ and $z \in B_{1/2}$, since $0\in \Gamma$.

Let $\bar u$ be the solution of \eqref{eq.ubareq} satisfying $\bar u = 0$ in $\R^n\setminus B_1$. By Proposition~\ref{prop.mainbound} we have that for each $t_0 \in (-1/2,0)$ there exists some $Q= Q(t_0,z)$ with $|Q| \leq C$ for which
\begin{equation}
\label{eq.mainthm1}
|u(t, x)-Q\bar{u}(x) | \leq C\left(|x-z|^{2s-\epsilon}+|t-t_0|^{\frac{2s-\epsilon}{2s}}\right)\textrm{ in } \left(-\frac{1}{2}+t_0,t_0\right)\times \R^n,
\end{equation}
where $C$ is a constant depending only on $n, \rho_0, s, \epsilon$ and the ellipticity constants \eqref{eq.L2}. We have used here Proposition \ref{prop.mainbound} on balls of radius $1/2$ around each $z\in \Gamma\cap B_{1/2}$. Notice that, $\bar u$ restricted to these balls satisfies an equation of the type \eqref{eq.ubareq} inside, and is bounded outside by $c_2$, so that Proposition~\ref{prop.mainbound} applies.

We will now divide the proof in two parts, concerning respectively the regularity for $u/d^s$ and the $(1-\epsilon)$-temporal regularity for $u$.
\\[0.5cm]
%***************************STEP 1******************************
{\bf Step 1: Regularity for $u/d^s$.}
To begin with, note that there is some $K = K(t_0, z)$ such that $|K| \leq C$ and
\begin{equation}
\label{eq.mainthm2}
\|u-K d^s\|_{L^\infty\left(\left(-r^{2s}+t_0,t_0\right)\times B_r(x_0)\right)} \leq Cr^{2s-\epsilon}.
\end{equation}
Indeed, this follows combining \eqref{eq.mainthm1} and \eqref{eq.rs1}, and assuming $r$ small enough.

On the other hand we also claim that
\begin{equation}
\label{eq.mainthm3}
[u-Kd^s]_{C^{\frac{1}{2} - \frac{\epsilon}{2s},s-\epsilon}_{t,x}((-r^{2s}+t_0,t_0)\times B_r(x_0))} \leq Cr^s,
\end{equation}
for any $t_0\in \left(-\frac{1}{2},0\right)$. Suppose also that it is always true that $-1 < t_0-r^{2s} < t_0 < 0$, making $r$ smaller if necessary. The constant $C$ in \eqref{eq.mainthm2}-\eqref{eq.mainthm3} depends only on $n, s, \epsilon, \rho_0$ and the ellipticity constants \eqref{eq.L2}.

To see \eqref{eq.mainthm3} define the following function
\begin{equation}
\label{eq.vrdef}
v_r(t, x) := r^{-s} u(r^{2s} t + t_0, r x + z) - r^{-s}Q\bar{u}(rx+z).
\end{equation}
Notice that by \eqref{eq.mainthm1} we have the following bound in $(-2^{2s},0)\times B_2(x_0)$,
\[
\|v_r\|_{L^\infty((-2^{2s},0)\times B_2)} \leq C r^{s-\epsilon},
\]
and that by \eqref{eq.mainthm1} we have the following growth control for $R \geq 1$,
\begin{equation}
\label{eq.mainthm5}
\sup_{t\in(-2,0)} \|v_r(t,\cdot)\|_{L^\infty( B_R)} \leq C r^{s-\epsilon}R^{2s-\epsilon}.
\end{equation}

Moreover, $v_r$ solves
\begin{equation}
\label{eq.mainthm7}
\de_t v_r - Lv_r = r^{s}\left(f(r^{2s}t+t_0,rx+z)-Q\right) \textrm{ in } (-2,0)\times B_2(\tilde{x_0}),
\end{equation}
for $\tilde{x_0} = \frac{x_0-z}{r}$, and $r$ small enough so that the domain in $t$ contains $(-2,0)$. Using the interior estimate in Lemma \ref{lem.bds0} and the bounds on $Q$, we obtain that $[v_r]_{C^{\frac{1}{2}-\frac{\epsilon}{2s},s-\epsilon}_{t, x}\left(\left(-1,0\right)\times B_1(\tilde{x_0})\right)} \leq Cr^{s-\epsilon}$. From this it follows that
\begin{equation}
\label{eq.mainthm8}
r^{s-\epsilon}[u-Q\bar{u}]_{C^{\frac{1}{2}-\frac{\epsilon}{2s},s-\epsilon}_{t, x}\left(\left(-r^{2s}+t_0,t_0\right)\times B_r({x_0})\right)} = r^s [v_r]_{C^{\frac{1}{2}-\frac{\epsilon}{2s},s-\epsilon}_{t, x}\left(\left(-1,0\right)\times B_1({\tilde{x_0}})\right)} \leq Cr^{2s-\epsilon},
\end{equation}
and so we get the desired result,  \eqref{eq.mainthm3}, by combining this expression with \eqref{eq.rs2}.

Finally, for any $x_1, x_2\in B_r(x_0)$ and any $t_1, t_2 \in \left(-r^{2s}+t_0,t_0 \right)$
\begin{align*}
\frac{u(t_1, x_1)}{d^s(x_1)} - \frac{u(t_2,x_2)}{d^s(x_2)} & =  \frac{(u-Kd^s)(t_1, x_1) - (u-Kd^s)(t_2, x_2)}{d^s(x_1)}~+ \\
&~~~+ (u-Kd^s)(t_2, x_2)(d^{-s}(x_1) - d^{-s}(x_2)).
\end{align*}

Now, by \eqref{eq.mainthm3} and using that $r$ and $d$ are comparable in $B_r(x_0)$ we have that
\[
\frac{|(u-Kd^s)(t_1, x_1) - (u-Kd^s)(t_2, x_2)|}{d^s(x_1)} \leq C\left(|x_1-x_2|^{s-\epsilon} + |t_1-t_2|^{\frac{1}{2}-\frac{\epsilon}{2s}}\right).
\]

By \eqref{eq.mainthm2} and \eqref{eq.rs3},
\[
|u-Kd^s|(t_2,x_2) |d^{-s}(x_1)-d^{-s}(x_2)| \leq C |x_1-x_2|^{s-\epsilon},
\]
we finally get that
\[
[u/d^s]_{C^{\frac{1}{2}-\frac{\epsilon}{2s},s-\epsilon}_{t, x}\left(\left(-r^{2s}+t_0,t_0\right)\times B_r({x_0})\right)} \leq C
\]
for all such balls $B_r(x_0)$. The bound
\[
 \left\| u/d^s \right\|_{C^{\frac{1}{2}-\frac{\epsilon}{2s},s-\epsilon}_x\left(\left(-\frac{1}{2},0\right)\times(\overline{\Omega^+}\cap {B_{1/2}})\right)} \leq C
\]
now follows the same way as in the proof of Proposition~\ref{prop.bound}.
\\[0.5cm]
%***************************STEP 2*****************************
{\bf Step 2: $C^{1-\epsilon}_t$ regularity.} Let us now prove the bound for $[u]_{C^{1-\epsilon}_{t}}$.

We begin by noticing that applying Lemma \ref{lem.bds0} to the solution $v_r$ of \eqref{eq.mainthm7} we have $[v_r]_{C^{1-\frac{\epsilon}{2s},2s-\epsilon}_{t,x}\left(\left(-1,0\right)\times B_1(\tilde{x_0})\right)} \leq Cr^{s-\epsilon}$, where $v_r$ was defined by \eqref{eq.vrdef}. Rescaling, we find
\begin{equation}
\label{eq.mainthm4}
r^{2s-\epsilon}[u- Q\bar u]_{C^{1-\frac{\epsilon}{2s},2s-\epsilon}_{t,x}\left(\left(-r^{2s}+t_0,t_0\right)\times B_r({x_0})\right)} = r^s [v_r]_{C^{1-\frac{\epsilon}{2s},2s-\epsilon}_{t,x}\left(\left(-1,0\right)\times B_1({\tilde{x_0}})\right)} \leq Cr^{2s-\epsilon}.
\end{equation}

From here, we deduce
\begin{equation}
\label{eq.mainthm6}
[u- Q (t_0,z) \bar u]_{C^{1-\frac{\epsilon}{2s},2s-\epsilon}_{t,x}\left(\left(-r^{2s}+t_0,t_0\right)\times B_r({x_0})\right)}  \leq C,
\end{equation}
where we will from now on explicitly write the dependence of $ Q$, and we remind that $r$ depends on $x_0$. Notice that from \eqref{eq.mainthm1},
\[
Q(t_0,z) = \lim_{\substack{z^*\in \Omega^+\\ z^*\to z}} \frac{u(t_0, z^*)}{\bar u (z^*)} =: \frac{u(t_0, z)}{\bar u (z)}.
\]
This last expression makes sense pointwise since the function $u/\bar u$ is continuous up to the boundary $\de\Omega^+\cap B_{1/2} = \Gamma \cap B_{1/2}$, so we can take the limit.

Indeed, we already proved that $\left\| u/d^s \right\|_{C^{\frac{1}{2}-\frac{\epsilon}{2s},s-\epsilon}_x\left(\left(-\frac{1}{2},0\right)\times(\overline{\Omega^+}\cap B_{1/2})\right)} \leq C$, and from \cite[Theorem 1.2]{RS} we know $\left\| \bar{u}/d^s \right\|_{C^{\frac{1}{2}-\frac{\epsilon}{2s},s-\epsilon}_x\left(\left(-\frac{1}{2},0\right)\times(\overline{\Omega^+}\cap B_{1/2})\right)} \leq C$: combining both expressions we obtain
\[
\|u/\bar u\|_{C^{\frac{1}{2}-\frac{\epsilon}{2s},s-\epsilon}_x\left(\left(-\frac{1}{2},0\right)\times(\overline{\Omega^+}\cap B_{1/2})\right)} \leq C.
\]

In particular, for any fixed $z$ on the boundary,
\begin{equation}
\label{eq.Qtimereg}
[Q(\cdot,z)]_{C_t^{\frac{1}{2}-\frac{\epsilon}{2s}}\left(-\frac{1}{2},0\right)} \leq C,
\end{equation}
where the constant $C$ does not depend on the point $z$ of the boundary inside $B_{1/2}$.

Now we want to show, for any point $y\in B_{1/4}$ and $t_1,t_2\in \left(-\frac{1}{2},0\right)$,
\[
|u(t_1,y)-u(t_2,y)|\leq C|t_1-t_2|^{1-\frac{\epsilon}{2s}},
\]
for some constant $C$ depending only on $\epsilon$, $n$, $s$, $\rho_0$ and the ellipticity constants \eqref{eq.L2}. Notice that to see this we can suppose that $|t_1-t_2|$ is as small as we need as long as its value does not depend on $y$. Let us consider $z$ such that
\[
 \textrm{dist}(y,\Gamma) = \textrm{dist}(y,z) < \rho_0.
\]
Notice that we can also assume that $|y-z|$ is as small as we need, since the interior regularity is already known.

If $|t_2-t_1| < 2^{-2s}|y-z|^{2s}$, then by \eqref{eq.mainthm6} we obtain the desired result. Assume now that
\begin{equation}
\label{eq.smallz}
|t_2-t_1| \geq 2^{-2s}|y-z|^{2s}.
\end{equation}

Let $\bar y$ to be chosen later, satisfying
\[
\textrm{dist}(\bar y, \Gamma) = \textrm{dist}(\bar y, z),
\]
i.e., in the line passing through $y$ and $z$.

Define $y_k = 2^{-k}\bar y + (1-2^{-k})y$, so that $y_0 = \bar y$, $y_\infty = y$. Define also
\[
w(t, x) := u(t, x) - Q(t,z) \bar u(x)
\]
which will be useful for points $x$ in the segment between $y$ and $\bar y$. With all this we can bound the following expression,
\begin{align*}
|w(t_1,y)-w(t_2,y)| & \leq \sum_{k\geq 0} |w(t_1,y_{k+1})-w(t_1,y_{k })| +\\
&~~~+ \sum_{k\geq 0} |w(t_2,y_{k+1})-w(t_2,y_{k })|  + |w(t_1,\bar y) - w(t_2,\bar y)| \\
& \leq 2\sum_{k\geq 0} C|y_{k+1}-y_{k}|^{2s-\epsilon}+ |w(t_1,\bar y) - w(t_2,\bar y)| \\
& =  2C|y-\bar y|^{2s-\epsilon}\sum_{k \geq 0} 2^{-(k+1)(2s-\epsilon)}+|w(t_1,\bar y) - w(t_2,\bar y)| \\
& \leq C|y-\bar y|^{2s-\epsilon} + |w(t_1,\bar y) - w(t_2,\bar y)|,
\end{align*}
Here, we used that
\[
y_{k+1}-y_{k} = 2^{-(k+1)} (y-\bar y),
\]
and therefore, in each term of the sum we can use the estimate \eqref{eq.mainthm6}. On the other hand
\begin{align*}
|w(t_1,\bar y) - w(t_2,\bar y)| &\leq |u(t_1,\bar y) - Q(t_1,z)\bar u (\bar y)-\left(u(t_2,\bar y) - Q(t_1,z)\bar u (\bar y)\right)|+\\
&~~~+|Q(t_1,z) - Q(t_2,z)|\bar u(\bar y).
\end{align*}

We take $\bar y$ such that
\begin{equation}
\label{eq.imp1}
|t_2-t_1| < 2^{-2s}|\bar y - z|^{2s},
\end{equation}
so that the estimate in \eqref{eq.mainthm6} is valid. Considering \eqref{eq.Qtimereg}, we have
\begin{align*}
|w(t_1,\bar y) - w(t_2,\bar y)| & \leq  C|t_2-t_1|^{1-\frac{\epsilon}{2s}} + C |t_2-t_1|^{\frac{1}{2}-\frac{\epsilon}{2s}} \bar u(\bar y)\\
& \leq  C|t_2-t_1|^{1-\frac{\epsilon}{2s}} + C |t_2-t_1|^{\frac{1}{2}-\frac{\epsilon}{2s}} |\bar y - z|^s;
\end{align*}
where in the last inequality we have used \eqref{eq.ubareq2}. Hence, using $|\bar y - z | > |\bar y - y|$,
\[
|w(t_1,y)-w(t_2,y)| \leq C(|\bar y-z|^{2s-\epsilon}+ |t_2-t_1|^{1-\frac{\epsilon}{2s}} +  |t_2-t_1|^{\frac{1}{2}-\frac{\epsilon}{2s}} |\bar y - z|^s).
\]
We now impose that $\epsilon_0^{2s}|\bar y - z|^{2s} \leq |t_2-t_1|$ for some $\epsilon_0$ independent of $y, \bar y, t_1$ and $t_2$, and we get
\[
|w(t_1,y)-w(t_2,y)| \leq C|t_2-t_1|^{1-\frac{\epsilon}{2s}}.
\]
Thus, we take $\bar y$ such that
\[
\epsilon_0|\bar y - z| \leq |t_2-t_1|^{\frac{1}{2s}} < 2^{-1}|\bar y - z|.
\]
This is always possible if $\epsilon_0$ and $|t_2-t_1|$ are small enough depending on $s$ and $\rho_0$.

Finally, from \eqref{eq.Qtimereg}, \eqref{eq.ubareq2} and \eqref{eq.smallz}
\begin{align*}
|u(t_1,y)-u(t_2,y)| & \leq |w(t_1,y)-w(t_2,y)| +|Q(t_1,z)-Q(t_2,z)|\bar u (y) \\
& \leq  C|t_2-t_1|^{1-\frac{\epsilon}{2s}} +  C|t_2-t_1|^{\frac{1}{2}-\frac{\epsilon}{2s}} |y-z|^{s} \\
& \leq C|t_2-t_1|^{1-\frac{\epsilon}{2s}},
\end{align*}
as we wanted to see.
\end{proof}

We finally give the:
\begin{proof}[Proof of Theorem \ref{thm.mainboundary}]
As in the proof of Proposition \ref{prop.mainboundary}, combine the result in Proposition~\ref{prop.mainbound2} with the interior estimates to get the desired result.
\end{proof}

\begin{rem}
\label{rem.csreg}
In a recent work by the second author and Serra, \cite{RS17}, the regularity results for the elliptic problem have been extended to $C^{1,\alpha}$ domains, for $\alpha\in(0,1)$.

In Section \ref{sec.4}, the fact that the domain is $C^{1,1}$ is only used in the construction of supersolutions in Lemma~\ref{lem.super}, and in Lemma~\ref{lem.bds}. Using that the solutions to the elliptic problem with a $C^{1,\alpha}$ domain are bounded by $d^s$, namely \eqref{eq.ds2}, then Proposition \ref{prop.mainboundary} is true for $C^{1,\alpha}$ domains, with the constant depending on $\alpha$ too. On the other hand, the argument done in Lemma~\ref{lem.super} can be easily adapted to $C^{1,\alpha}$ domains.

In Section \ref{sec.5} there are two steps where we used the $C^{1,1}$ regularity of the domain. Namely, to obtain the bounds \eqref{eq.ubareq2} from \cite{RS}, and to say that $\frac{\bar{u}(rx)}{r^s}$ converges uniformly in $B_1$ as $r\downarrow 0$ to $\bar{Q}(x_n)^s_+$ in \eqref{eq.unifconv}. Again, using the elliptic results for a domain $C^{1,\alpha}$ with $\alpha \in (0,1)$, then the regularity up to the boundary found for $u/d^s$ follows as well for this class of domains.
\end{rem}

\section{The Dirichlet problem}
\label{sec.6}
In this section we prove Corollary~\ref{cor.mainboundary}. First, we give the following lemma.

\begin{lem}
\label{lem.u0_L2norm}
Let $s\in(0,1)$ and let $\Omega\subset \R^n$ be a bounded $C^{1,1}$ domain satisfying the exterior ball condition. Let $u$ be the solution to
\begin{equation}\label{eq.fracheat3}
  \left\{ \begin{array}{rcll}
  \de_t u - L u&=&c_0& \textrm{in }\Omega,\ t > 0 \\
  u&=&0& \textrm{in } \R^n\setminus \Omega,\ t \geq 0 \\
  u(0,x)&=&u_0&\textrm{in }\Omega,
  \end{array}\right.
\end{equation}
for some constant $c_0 > 0$. Then, we have
\begin{equation}
|u| \leq C(t_0) (\|u_0\|_{L^2(\Omega)}+c_0) d^s,~~\textrm{ for all } t\geq t_0 > 0,
\end{equation}
where $C(t_0)$ depends only on $t_0, n, s, \Omega$ and the ellipticity constants \eqref{eq.L2}. The dependence with respect to $\Omega$ is via $|\Omega|$ and the $C^{1,1}$ norm of the domain.
\end{lem}
\begin{proof}
Proceed exactly as in the proof of Lemma \ref{lem.super}.
\end{proof}

And now we can prove Corollary \ref{cor.mainboundary}.

\begin{proof}[Proof of Corollary \ref{cor.mainboundary}]
For the first part, cover $\Omega$ by a finite number of unit balls and apply Theorem \ref{thm.mainb} to the interior balls and Proposition \ref{prop.bound} to balls with center on the boundary, to get
\begin{equation}
\|u\|_{C^{\frac{1}{2},s}_{t,x}\left(\left(\frac{1}{2},1\right)\times\overline{\Omega}\right)}\leq C\left(\|u\|_{L^\infty\left(\left(\frac{1}{4},1\right)\times\Omega\right)}+\|f\|_{L^\infty\left(\left(\frac{1}{4},1\right)\times\Omega\right)} \right).
\end{equation}
Similarly, applying Proposition \ref{prop.mainbound2} we get
\begin{equation}
\|u\|_{C^{1-\epsilon}_{t}\left(\left(\frac{1}{2},1\right)\times\overline{\Omega}\right)}+\|u/d^s\|_{C^{\frac{1}{2}-\frac{\epsilon}{2s},s-\epsilon}_{t,x}\left(\left(\frac{1}{2},1\right)\times\overline{\Omega}\right)}\leq C\left(\|u\|_{L^\infty\left(\left(\frac{1}{4},1\right)\times\Omega\right)}+\|f\|_{L^\infty\left(\left(\frac{1}{4},1\right)\times\Omega\right)} \right).
\end{equation}

On the other hand, by Lemma \ref{lem.u0_L2norm} with $c_0 = \|f\|_{L^\infty\left(\left(0,1\right)\times\Omega\right)}$ and $t = t_0 = 1/4$,
\begin{align*}
\|u\|_{L^\infty\left(\left(\frac{1}{4},1\right)\times\Omega\right)} & \leq  C\left(\|u(1/4,\cdot)\|_{L^\infty\left(\Omega\right)}+\|f\|_{L^\infty\left(\left(\frac{1}{4},1\right)\times\Omega\right)}\right)\\
& \leq  C\left(\|u(0,\cdot)\|_{L^2\left(\Omega\right)}+\|f\|_{L^\infty\left(\left(0,1\right)\times\Omega\right)}\right).
\end{align*}

Finally, combining the previous expressions and rescaling the temporal domain appropriately we get the desired result.

For the second part it is enough to combine the previous result with the interior regularity estimates \eqref{eq.maina}. This can be done as long as $\alpha \leq s$, from the $C^{\frac{1}{2},s}_{t,x}$ estimate.
\end{proof}

Let us finish by proving a corollary with sufficient conditions on $f$ for $u$ to have classical derivatives with respect to time $t$ up to the boundary.

\begin{cor}
\label{cor.main2_2}
Let $s\in(0,1)$, let $L$ be any operator of the form \eqref{eq.L1}-\eqref{eq.L2} and let $\Omega$ be a bounded $C^{1,1}$ domain. Let $f\in C^\delta_t ((0,1)\times \Omega)$ for some $\delta > 0$, $u_0\in L^2(\Omega)$, and let $u$ be the weak solution to
\begin{equation}\label{eq.fracheat_u0_cor}
  \left\{ \begin{array}{rcll}
  \de_t u - L u&=&f& \textrm{in }\Omega,\ t > 0 \\
  u&=&0& \textrm{in } \R^n\setminus \Omega,\ t \geq 0, \\
  u(0,\cdot)&=&u_0& \textrm{in } \Omega,\ t = 0. \\
  \end{array}\right.
\end{equation}

Then,
\begin{equation}
\|\de_t u\|_{L^\infty\left(\left(\frac{1}{2},1\right)\times\overline{\Omega}\right)} +\|Lu\|_{L^\infty\left(\left(\frac{1}{2},1\right)\times\overline{\Omega}\right)}  \leq C \left( \|u_0\|_{L^2(\Omega)} + \|f\|_{C^\delta_t((0,1)\times \Omega)} \right),
\end{equation}
where the constant $C$ depends only on $ \delta, n, s, \Omega$ and the ellipticity constants \eqref{eq.L2}.
\end{cor}

\begin{proof}
We proceed as in Corollary \ref{cor.main1} by taking incremental quotients in $t$,
\[
u^\tau_{\delta}(t+\tau, x) := \frac{u(t+\tau, x)-u(t, x)}{|\tau|^{\delta}},~~~~f^\tau_{\delta}(t, x) := \frac{f(t+\tau, x)-f(t, x)}{|\tau|^{\delta}}
\]
for some $\tau > 0$ fixed. Notice that
\[
\de_t u^\tau_{\delta} - L u^\tau_{\delta} = f^\tau_{\delta}~ \textrm{ in }~ \Omega, t > \frac{1}{4}.
\]

Note also that
\begin{align*}
\|f^\tau_{\delta}\|_{L^\infty\left(\left(\frac{1}{4}+\tau,1\right)\times \Omega\right)}  & \leq C\left( [f]_{C^{\delta}_t \left(\left(\frac{1}{4},1\right)\times \Omega\right) } \right),\\
\|u^\tau_{\delta}\|_{L^\infty\left(\left(\frac{1}{4}+\tau,1\right)\times \Omega\right)} & \leq C\left( [u]_{C^{\delta}_t \left(\left(\frac{1}{4},1\right)\times \Omega\right) } \right) \leq C \left( \|u_0\|_{L^2(\Omega)} + \|f\|_{L^\infty((0,1)\times \Omega)} \right),
\end{align*}
where in the last inequality we have used the first part of Corollary~\ref{cor.mainboundary}.

On the other hand,
\[
 \sup_{0 \leq \tau \leq \frac{1}{16}} \|u^\tau_{\delta}\|_{C^{1-\epsilon}_t\left(\left(\frac{1}{2},1\right)\times\Omega\right)} \geq   C \|u\|_{C^{1-\epsilon+\delta}_t\left(\left(\frac{1}{2},1\right)\times\Omega\right)},
\]
(see, for example, \cite[Lemma 5.6]{CC}). Combining the previous results with Corollary \ref{cor.mainboundary} we obtain the following bound,
\begin{equation}
\|u\|_{C^{1+\delta-\epsilon}_t\left(\left(\frac{1}{2},1\right)\times \overline{\Omega}\right)} \leq C \left( \|u_0\|_{L^2(\Omega)} + \|f\|_{C^\delta_t((0,1)\times \Omega)} \right).
\end{equation}
Using that $\|Lu\|_{L^\infty\left(\left(\frac{1}{2},1\right)\times\overline{\Omega}\right)} \leq \|\de_t u\|_{L^\infty\left(\left(\frac{1}{2},1\right)\times\overline{\Omega}\right)} + \|f\|_{L^\infty\left(\left(\frac{1}{2},1\right)\times\overline{\Omega}\right)} $, we are done.
\end{proof}

We finally give a corollary on the higher regularity in $t$.
\begin{cor}
\label{cor.main2_3}
Let $s\in(0,1)$, let $L$ be any operator of the form \eqref{eq.L1}-\eqref{eq.L2} and let $\Omega$ be any bounded $C^{1,1}$ domain. Let $f\in C^k_t ((0,1)\times \Omega)$ for some $k\in\N$, $u_0\in L^2(\Omega)$, and let $u$ be the weak solution to \eqref{eq.fracheat_u0_cor}. Then, for any $\delta > 0$
\begin{equation}
\| u\|_{C^k_t\left(\left(\frac{1}{2},1\right)\times\overline{\Omega}\right)}  \leq C \left( \|u_0\|_{L^2(\Omega)} + \|f\|_{C^{k-1,\delta}_t((0,1)\times \Omega)} \right),
\end{equation}
where the constant $C$ depends only on $\delta, k, n, s, \Omega$ and the ellipticity constants \eqref{eq.L2}.
\end{cor}

\begin{proof}
If $k = 1$, the result is a consequence of the previous statement, Corollary~\ref{cor.main2_2}. Now we proceed by induction.

Suppose it is true for $k = q$, and let us check we can obtain the same for $q+1$. Notice that this means $v_q:=\de_t^{(q)}u$ is bounded and therefore classically defined, and it solves
\begin{equation}
  \left\{ \begin{array}{rcll}
  \de_t v_q - L v_q&=&\de_t^{(q)}f& \textrm{in }\Omega,\ t > 0 \\
  v_q&=&0& \textrm{in } \R^n\setminus \Omega,\ t \geq 0. \\
  \end{array}\right.
\end{equation}

Using incremental quotients, as in the proof of Corollary~\ref{cor.main2_2}, we obtain
\begin{equation}
\|\de_t v_q\|_{L^\infty\left(\left(\frac{1}{2},1\right)\times\overline{\Omega}\right)}  \leq C \left( \|v_q(1/4,\cdot)\|_{L^\infty(\Omega)} + \|\de_t^{(q)} f\|_{C^\delta_t((1/4,1)\times \Omega)} \right).
\end{equation}
From the induction hypothesis
\[
\| v_q\|_{L^\infty\left(\left(\frac{1}{8},1\right)\times\overline{\Omega}\right)}  \leq C \left( \|u_0\|_{L^2(\Omega)} + \|f\|_{C^{q-1,\delta}_t((0,1)\times \Omega)} \right),
\]
and combining the previous two expressions we are done.
\end{proof}
\begin{rem}
Thanks to the previous result, if $f\in C^\infty_t$ up to the boundary then the solution to the Dirichlet problem is automatically $C^\infty_t$ inside $\Omega$. As explained in the introduction, this does not happen in space for general stable operators.

Still, if $f\in C^\infty_{x}$ and the kernel satisfies $a\in C^\infty(S^{n-1})$, then by Corollary~\ref{cor.main3} the solution $u$ is $C^\infty_{x}$ in the interior of the domain.
\end{rem}

\section{Sharpness of the estimates}
\label{sec.7}
In this final section we discuss the sharpness of the estimates in Theorems~\ref{thm.maina} and~\ref{thm.mainboundary}.

\subsection{Sharpness for the interior estimates}

Solutions to the elliptic problem are, in particular, solutions to the parabolic problem. In \cite[Proposition 6.1]{RS}, the second author and Serra proved that to gain up to $C^{\alpha+2s}_x$ spatial regularity it is necessary to assume that the solution is, at least, in $C^\alpha_x$.

Indeed, they construct a function $u:\R^2\to\R$ such that, for $\alpha \in (0,s]$, $\epsilon>0$ small, and a certain operator of the form \eqref{eq.L1}-\eqref{eq.L2},
\begin{enumerate}[(i)]
\item $Lu = 0$ in $B_1$
\item $u = 0$ in $B_2\setminus B_1$
\item $u \in C^{\alpha-\epsilon}(\R^2)$
\item $u\notin C^{\alpha+2s}(B_{1/2})$.
\end{enumerate}

This proves the sharpness of the bounds for Theorem \ref{thm.maina} in the spatial part, since the condition on the a priori regularity on $u$ is sharp.

We next show that the temporal part is also optimal: from a function $u$ that is $\frac{\alpha}{2s}$-Hölder in $t$, we can gain regularity up to $1+\frac{\alpha}{2s}$ and not better in general. This follows from the following lemma, that uses a construction inspired by \cite[Counterexample 2.4.1]{CD}.

\begin{lem}
Let $s \in (0,1)$, and let $\alpha \in (0,s]$. Then, for any $\epsilon > 0$ small, there exists a function $v$ satisfying:
\begin{enumerate}[(i)]
\item $\de_t v + \fls v = 0$ in $(-1,0)\times B_1$
\item $v \in C^{\frac{\alpha}{2s}-\epsilon, \alpha}_{t,x}((-1,0)\times \R^n)$
\item $v\notin C^{1+\frac{\alpha}{2s}}_{t}\left(\left(-\frac{1}{2},0\right)\times B_{1/2}\right)$ for any $\epsilon > 0$.
\end{enumerate}
\end{lem}

\begin{proof}
Consider the solution $v$ to the fractional heat equation
\[
\de_t v +\fls v = 0 \textrm{ in } (-1,0)\times B_1
\]
with initial value $v(-1,\cdot)\equiv 0$ and exterior condition outside $B_1$ equal to $\bar v$,
\begin{equation}
  \bar v(t, x) = \left\{ \begin{array}{ll}
  0& \textrm{if }t < -\frac{1}{4} \\
  \epsilon_0 \left(t+\frac{1}{4}\right)^{1+\delta} + (t+1/4)^{\delta}\eta(x) &\textrm{if } t \geq -\frac{1}{4}, \\
  \end{array} \right.
\end{equation}
where we fix $\delta = \frac{\alpha}{2s}-\epsilon>0$, and $\eta(x)$ is a $C^\infty$ non-negative function supported in $B_{7/2}\setminus B_{3/2}$ and equal to 1 in $B_3\setminus B_2$.

On the one hand, $v \in C^{\frac{\alpha}{2s}-\epsilon, \alpha}_{t,x}((-1,0)\times \R^n)$. Indeed, for times in $(-1,-1/4)$, $v\equiv 0$ by uniqueness. For times $t\geq -1/4$, this is true inside $B_1$ by the $C^s$ regularity up to the boundary, Proposition~\ref{prop.mainboundary}; and it is also true outside the ball by the regularity of the exterior condition.

On the other hand, for $\epsilon_0 > 0$ small, then $v$ is at most $C^{1+\delta}_t$ in $(-1/2,0)\times B_{1/2}$. To see this, notice that, for $t \geq -1/4$ and in $B_{1}$,
\[
\de_t \bar v + \fls \bar v = (t+1/4)^\delta\left((1+\delta)\epsilon_0 + \fls \eta\right).
\]
Note also that $\fls \eta \leq -c$ in $B_{1/2}$ for some positive constant $c$, so that we can choose $\epsilon_0 >0$ small enough such that $\bar v$ is a subsolution to the fractional heat equation in $(-1,0)\times B_1$. By the comparison principle, $v \geq \bar v$ in $(-1,0)\times B_1$, and also $v \equiv 0$ in $(-1,-1/4)\times B_1$ by uniqueness. Thus, $\de_t v $ is at most $C^\delta_t$ inside $B_{1/2}$, and $\delta < \frac{\alpha}{2s}$.
\end{proof}

\subsection{Sharpness for the boundary estimates}

The $C^s_x$ regularity up to the boundary for the solutions to nonlocal parabolic equations is optimal (as it is optimal even for the fractional Laplacian in the elliptic case).

Regarding the optimality of the bounds for the estimates up to the boundary for $u/d^s$ we expect them to be optimal or almost optimal for general $f\in L^\infty$ (even for the fractional Laplacian) because the regularity cannot exceed the one achieved in the interior.

For general stable operators we expect this regularity to be optimal even if $\Omega$ is a $C^\infty$ domain and $f\in C^\infty$. We refer to \cite[Proposition 6.2]{RS}, where it is proven that for some operator $L$ of the form \eqref{eq.L1}-\eqref{eq.L2} and some $C^\infty$ domain $\Omega$, one has $L(d^s)\notin L^\infty(\Omega)$. Thus, we do not expect to have $C^s$ regularity up to the boundary for the quotient $u/d^s$.


\begin{thebibliography}{00}
\bibitem[CC95]{CC} L. Caffarelli, X. Cabré, \emph{Fully Nonlinear Elliptic Equations}, AMS Colloquium Publications, Vol 43, 1995.
\bibitem[CD14]{CD} H. Chang-Lara, G. Dávila, \emph{Regularity for solutions of nonlocal parabolic equations II}, J. Differential Equations 256 (2014), 130-156.
\bibitem[CK15]{CLK} H. Chang-Lara, D. Kriventsov, \emph{Further time regularity for non-local, fully non-linear parabolic equations}, Comm. Pure Appl. Math., to appear.
\bibitem[CKS10]{CKS} Z. Chen, P. Kim, R. Song, \emph{Heat kernel estimates for the Dirichlet fractional Laplacian}, J. Eur. Math. Soc. 12 (2010), 1307-1329.
%\bibitem[DK15]{DK} B. Dyda, M. Kassmann, \emph{Regularity estimates for elliptic nonlocal operators}, preprint arXiv (2015).
\bibitem[FR15]{FR} X. Fernández-Real, X. Ros-Oton, \emph{Boundary regularity for the fractional heat equation},
Rev. Acad. Cienc. Ser. A Math. 110 (2016), 49-64.
\bibitem[Gei14]{G} L. Geisinger, \emph{A short proof of Weyl's law for fractional diferential operators}, J. Math. Phys.
55 (2014), 011504.
\bibitem[Gru15]{GG15} G. Grubb, \emph{Fractional Laplacians on domains, a development of Hörmander's theory of $\mu$-transmission pseudodifferential operators}, Adv. Math. 268 (2015), 478-528.
\bibitem[Gru14]{GG14} G. Grubb, \emph{Local and nonlocal boundary conditions for $\mu$-transmission and fractional elliptic pseudodifferential operators}, Anal. PDE 7 (2014), 1649-1682.
\bibitem[JX15]{JX} T. Jin, J. Xiong, \emph{Schauder estimates for solutions of linear parabolic integro-differential equations}, Discrete Contin. Dyn. Syst. A. 35 (2015), 5977-5998.
\bibitem[KS14]{KS} M. Kassmann, R. W. Schwab \emph{Regularity results for nonlocal parabolic equations},  Riv. Mat. Univ. Parma 5 (2014), 183-212.
\bibitem[Ros15]{R} X. Ros-Oton, \emph{Nonlocal elliptic equations in bounded domains: a survey}, Publ. Mat. 60 (2016), 3-26.
\bibitem[RS14]{RS2} X. Ros-Oton, J. Serra, \emph{The Dirichlet problem for the fractional Laplacian: regularity up to the boundary}, J. Math. Pures Appl. 101 (2014), 275-302.
\bibitem[RS14b]{RS} X. Ros-Oton, J. Serra, \emph{Regularity theory for general stable operators}, J. Differential Equations 260 (2016), 8675-8715.
\bibitem[RS17]{RS17} X. Ros-Oton, J. Serra, \emph{Boundary regularity estimates for nonlocal elliptic equations in $C^1$ and $C^{1,\alpha}$ domains}, Ann. Mat. Pura Appl., to appear.
\bibitem[RV15]{RV} X. Ros-Oton, E. Valdinoci, \emph{The Dirichlet problem for nonlocal operators with singular kernels: convex and nonconvex domains}, Adv. Math. 288 (2016), 732-790.
\bibitem[ST94]{ST} G. Samorodnitsky, M. S. Taqqu, \emph{Stable Non-Gaussian Random Processes: Stochastic Models
With Infinite Variance}, Chapman and Hall, New York, 1994.
\bibitem[SS14]{SS} R. Schwab, L. Silvestre, \emph{Regularity for parabolic integro-differential equations with very irregular kernels}, Anal. PDE 9 (2016), 727-772.
\bibitem[Ser15]{Ser} J. Serra, \emph{Regularity for fully nonlinear nonlocal parabolic equations with rough kernels}, Calc. Var. Partial Differential Equations 54 (2015), 615-629.
\bibitem[SV13]{SV} R. Servadei, E. Valdinoci, \emph{A Brezis-Nirenberg result for non-local critical equations in low dimension}, Comm. Pure Appl. Anal. 12 (2013), 2445-2464.
\bibitem[Sim97]{S} L. Simon, \emph{Schauder estimates by scaling}, Calc. Var. Partial Differential Equations 5 (1997), 391-407.


\end{thebibliography}
\end{document}